%% file: main_ACOM.tex
\RequirePackage{etex}
\documentclass[pdflatex,sn-mathphys-num]{sn-jnl}

\usepackage{graphicx}
\usepackage{multirow}
\usepackage{amsmath,amsfonts,amssymb}
\usepackage{mathrsfs}
\usepackage{xcolor}
\usepackage{textcomp}
\usepackage{manyfoot}
\usepackage{booktabs}
\usepackage{listings}

\usepackage[english]{babel}
\usepackage{csquotes}
\usepackage{microtype}
\usepackage{nicefrac}
\usepackage{comment}
\usepackage{siunitx}

\usepackage{tikz,tikz-network,pgfplots}
\pgfplotsset{compat=1.18}
\usepgfplotslibrary{groupplots}

\usepackage[caption=false]{subfig}

\usepackage[ruled,vlined,linesnumberedhidden]{algorithm2e}

\definecolor{Hylinkcolor}{HTML}{800006}
\definecolor{Hycitecolor}{HTML}{2E7E2A}
\definecolor{Hyfilecolor}{HTML}{131877}
\definecolor{Hyurlcolor}{HTML}{8A0087}
\definecolor{Hymenucolor}{HTML}{727500}
\definecolor{Hyruncolor}{HTML}{137776}

\hypersetup{
  colorlinks=true,
  linkcolor=Hylinkcolor,
  citecolor=Hycitecolor,
  filecolor=Hyfilecolor,
  urlcolor=Hyurlcolor,
  menucolor=Hymenucolor,
  runcolor=Hyruncolor
}

\definecolor{blind1}{HTML}{d55e00}
\definecolor{blind2}{HTML}{cc79a7}
\definecolor{blind3}{HTML}{0072b2}
\definecolor{blind4}{HTML}{f0e442}
\definecolor{blind5}{HTML}{009e73}

\definecolor{mycolor1}{rgb}{0.86600,0.32900,0.00000}
\definecolor{mycolor2}{rgb}{0.92900,0.69400,0.12500}
\definecolor{mycolor3}{rgb}{0.12941,0.12941,0.12941}

\theoremstyle{thmstyleone}
\newtheorem{theorem}{Theorem}
\newtheorem{proposition}[theorem]{Proposition}
\newtheorem{lemma}[theorem]{Lemma}
\newtheorem{corollary}[theorem]{Corollary}

\theoremstyle{thmstyletwo}

\newtheorem{remark}{Remark}

\theoremstyle{thmstylethree}
\newtheorem{definition}{Definition}

\begin{document}

\title[Steering dynamic network centrality]{Steering dynamic network centrality via control theory}

\author*[1]{\fnm{Fabio} \sur{Durastante}}\email{fabio.durastante@unipi.it}
\equalcont{These authors contributed equally to this work.}

\author[1]{\fnm{Beatrice} \sur{Meini}}\email{beatrice.meini@unipi.it}
\equalcont{These authors contributed equally to this work.}

\author[2]{\fnm{Luca} \sur{Saluzzi}}\email{luca.saluzzi@unife.it}
\equalcont{These authors contributed equally to this work.}

\affil*[1]{\orgdiv{Department of Mathematics}, \orgname{University of Pisa}, \orgaddress{\street{Largo Bruno Pontecorvo, 5}, \city{Pisa}, \postcode{56127}, \country{Italy}}}

\affil[2]{\orgdiv{Department of Mathematics and Informatics}, \orgname{University of Ferrara}, \orgaddress{\street{Via Nicolò Machiavelli, 30}, \city{Ferrara}, \state{44121}, \country{Italy}}}

\abstract{Time-evolving networks, or temporal networks, play a crucial role in modeling dynamic interactions across various domains, including biology, social sciences, and information technology. Unlike static networks, these systems undergo continuous changes in topology and edge weights, influencing processes such as information flow, transportation efficiency, and neural activity. Understanding and controlling these networks are essential for predicting future behavior and optimizing dynamic processes. This work focuses on the problem of dynamic centrality, a measure of node importance in time-dependent networks. Specifically, we address how to steer network centrality to a desired state by making minimal modifications to the network structure. This problem is formulated as an optimal control problem for an ordinary differential equation, either matrix- or vector-based, where the control acts on network edges. The proposed framework generalizes centrality control problems studied in static networks and leverages the Pontryagin Maximum Principle for efficient solutions. For large-scale problems, the required matrix-function actions are approximated by Krylov-type techniques, avoiding the explicit formation of dense matrix functions. Numerical experiments on synthetic and real temporal networks show that the proposed framework can effectively steer receive centrality under prescribed control constraints.
}

\keywords{Temporal network, centrality measure, dynamic receive, control theory}

\pacs[MSC Classification]{05C50, 49M05, 15A16}

\maketitle

\section{Introduction}%

Complex networks have become a fundamental tool for modeling and understanding the structure and dynamics of interconnected systems across various domains, such as biology~\cite{Goh20078685}, social sciences, and information technology~\cite{Vespignani2009425}. In many cases, networks are considered static entities, where the set of nodes and edges remains unchanged over time. However, numerous real-world networks evolve dynamically, with both topology and edge weights varying over time. Such networks are referred to as \emph{time-evolving} or \emph{temporal} networks~\cite{Holme201297}. These dynamic interactions capture a wide range of phenomena, including communication patterns in social networks, where two users may follow each other for a certain period before one of them stops following the other. Similarly, in transportation systems, flight routes between airports may only be active during specific months of the year~\cite{PhysRevE.84.016105}. In brain networks, neural pathways may be activated only under particular conditions~\cite{Bullmore2009,doi:10.1073/pnas.1018985108}. Understanding how these networks evolve is crucial for uncovering patterns, predicting future behavior, and controlling processes occurring within the network.

The study of time-evolving networks has led to the development of an array of methods for analyzing how information propagates through such systems. Concepts such as dynamic centrality measures~\cite{HighamNet,10.1145/1830252.1830262,MR4490469,10.1093/comnet/cnw017}, temporal motifs~\cite{10.1145/3018661.3018731}, and communicability matrices~\cite{PhysRevE.83.046120} extend classical static graph metrics~\cite{MR2736969} to the temporal domain. These tools enable a more accurate and comprehensive understanding of the roles of individual nodes and overall network behavior in time-varying structures. In this context, we focus on how to control network connections by adding, modifying, or removing edges to drive measures of network communicability toward a desired state. This work addresses the centrality problem in dynamic networks, specifically \emph{dynamic centrality}~\cite{HighamNet}. This measure computes the centrality of a node by solving a matrix differential equation when computing communicability indices between any two nodes or a vector differential equation when determining dynamic broadcast or receive centrality. Our goal is to determine a minimal modification of the time-dependent network that achieves a prescribed centrality at a given final time $T$. This generalizes the analogous problem for centralities in static networks~\cite{BenziGuglielmi,cipolla2025enforcingkatzpagerankcentrality,MR4822696}. Since dynamic centrality is formulated as the solution of an ordinary differential equation (ODE), the problem can be posed as an optimal control problem for an ODE, either matrix- or vector-based, where the control acts on network edges to steer the centrality toward the desired values.

Control theory has been widely applied to networks in various contexts. One example is the introduction of a dynamical process defined on network edges~\cite{Nepusz2012}, demonstrating that the controllability properties of such edge-based dynamics differ significantly from those of traditional node-based dynamics. In the latter case, it is possible to identify a set of driver nodes with time-dependent control inputs that can steer the entire system's dynamics~\cite{Liu2011}. This edge-based control framework is particularly relevant for monitoring the evolution of dynamic phenomena on networks, including epidemic spread, opinion dynamics and other compartmental models~\cite{9903909,Gomes,albi2026controlkineticopiniondynamics}. 
Recent work has further addressed controllability in dynamic networks by proposing an adaptive control algorithm that minimizes driver node switching costs under stochastic topology changes without requiring prior knowledge of network evolution~\cite{11150536}; while in the same direction also a \emph{graphical criterion} to evaluate the controllable subspace of temporal networks via maximum matching on aggregated static representations has been introduced~\cite{TU2024129906}, and another approach adopts a temporal ant colony optimization–based approach to efficiently approximate maximum matching in temporal networks to identify driver nodes for structural controllability~\cite{Xia02092025}.

For the problems we consider here, we have either a nonlinear or a linear ODE whose state variable represents the time-dependent centrality measure. Hence, the problem we face is that of steering this state towards a prescribed value by modifying the time-dependent edge weights. This is indeed an instance of a bilinear control problem in which the control, or more precisely one of its functions, acts multiplicatively on the state of the system. As it is often the case, the control problem is highly dimensional, and due to the underlying state equation~\cite{HighamNet} involves the computation of large-scale matrix-function~\cite{HighamBook}, and matrix-function vector products~\cite{BerljafaGuettel2015,GuettelKnizhnerman2013}. We discuss and adapt the related techniques for working in our regime of interest---large scale, repeated computation, possibly non-normal matrices. To formulate the solution of the optimal control problem we make use of the Pontryagin Maximum Principle (PMP) which requires computing the evolution of the adjoint variable, due to the presence of the matrix function in the state equation this requires the computation of large-scale Fr\'echet derivatives of a matrix function~\cite{MR3080997}, and again their application against a vector~\cite{FrechetTrick} through the construction of suitable rational Krylov methods.
Henceforth, the main contributions of this work include a pattern-constrained optimal-control formulation for steering dynamic receive centrality in temporal networks and the derivation of sufficient linear feasibility constraints to preserve the well-posedness of the logarithmic model. Furthermore, we establish PMP-based first-order optimality conditions for both logarithmic and linear dynamics and present a matrix-free projected-gradient algorithm that utilizes Krylov and rational Krylov approximations for matrix logarithm and Fréchet derivative actions. Finally, we provide empirical validation of our approach using both synthetic and real temporal networks.
The manuscript is organized as follows. In Section~\ref{sec:notation_and_basic_definition}, we introduce the notation and context necessary to describe time-dependent networks. Section~\ref{sec:dynamical_system} presents the construction of the dynamic centrality measure we focus on~\cite{HighamNet}, which serves as the foundation for formulating the problem of steering centrality toward a desired state in Section~\ref{sec:the-steering-problem}. Once the problem is formulated, we discuss its solution using the PMP in Section~\ref{sec:PMP}, followed by a description of its discretization and efficient implementation in Section~\ref{sec:discretizing_and_solving}. Section~\ref{sec:numerical_examples} applies the proposed methodology to synthetic and real-world temporal network test cases. Finally, in Section~\ref{sec:conclusion_and_future_perspectives}, we conclude our analysis and discuss potential extensions.

\subsection{Notation and basic definitions}\label{sec:notation_and_basic_definition}

\begin{definition}\label{def:time-evolving-graph}
A \emph{time-evolving weighted network (or graph)} is a dynamic graph that evolves over time, formally represented as a tuple $G_t = (V, E_t, w_t)$, where $V$ is the set of vertices, which remains constant over time, $E_t \subseteq V \times V$ is the set of edges at time $t$, for $t \in \mathbb{T}$ and $\mathbb{T}$ the discrete or continuous time domain; $w_t: E_t \to \mathbb{R}^+$ is a function that assigns a weight to each edge at time $t$, representing the strength, capacity, or cost of the connection between vertices; see Figure~\ref{fig:time-evolving-graph} for an example.
\begin{figure}[htbp]
    \centering
\begin{center}
\begin{tikzpicture}[scale=0.8]
\begin{scope}[xshift=-4cm]
\Vertex[size=0.3,color=white]{A} \Vertex[x=2,size=0.3,color=white]{B} \Vertex[y=-2,size=0.3,color=white]{C}
\Edge[label=$1$](A)(B)
\Edge[label=$2$](B)(C)
\Text[x=1, y=0.6]{$t_1$}
\end{scope}
\begin{scope}[xshift=0cm]
\Vertex[size=0.3,color=white]{A} \Vertex[x=2,size=0.3,color=white]{B} \Vertex[y=-2,size=0.3,color=white]{C}
\Edge[label=$1$](A)(B)
\Edge[label=$3$](A)(C)
\Text[x=1, y=0.6]{$t_2$}
\end{scope}
\begin{scope}[xshift=4cm]
\Vertex[size=0.3,color=white]{A} \Vertex[x=2,size=0.3,color=white]{B} \Vertex[y=-2,size=0.3,color=white]{C}
\Edge[label=$2$](A)(B)
\Edge[label=$3$](B)(C)
\Text[x=1, y=0.6]{$t_3$}
\end{scope}
\end{tikzpicture}
\end{center}
\caption{Three-node temporal network with weighted edges over three time steps.\label{fig:time-evolving-graph}}
\end{figure}
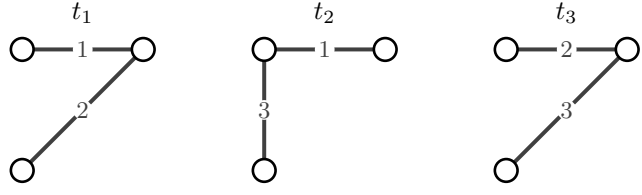
If for all $t \in \mathbb{T}$ all vertices $(v_i,v_j) \in E_t$ are reciprocated, then the network is said to be undirected, otherwise it is directed.
\end{definition}

The edge set $E_t$ and weight function $w_t$ are both time-dependent, reflecting the fact that the structure and properties of the graph may change as time progresses. To represent it in a compact form we define the related concept of a time-dependent (weighted) adjacency matrix.

\begin{definition}\label{def:time-adjacency}
Let $G_t = (V, E_t, w_t)$ be a time-evolving weighted graph, where $V = \{v_1, v_2, \dots, v_n\}$ is a set of $n$ vertices. The \emph{adjacency matrix} $A(t) \in \mathbb{R}^{n \times n}$ of the graph $G_t$ at time $t$ is a square matrix, where each element $(A(t))_{ij}$ is defined as:
\[
(A(t))_{ij} = 
\begin{cases} 
w_t(v_i, v_j) & \text{if } (v_i, v_j) \in E_t, \\
0 & \text{if } (v_i, v_j) \notin E_t,
\end{cases}
\]
where $w_t(v_i, v_j)$ is the weight of the edge between vertices $v_i$ and $v_j$ at time $t$. For an undirected network $A(t)$ is a symmetric matrix for all $t \in \mathbb{T}$. 
\end{definition}

Throughout the paper we assume that all $G_t$ are such that $A(t) \leq \overline{A}$ for all $t \in \mathbb{T}$. This boundedness property guarantees that the spectral radius of all $A(t)$ is bounded as well, and will permit formulating the differential models in Section~\ref{sec:dynamical_system}. We also note that, in many applications, the network $G_t$ is observed only at a finite collection of time points $0 = t_1 < t_2 < \cdots < t_k = T$. In this setting, we still take the time domain to be $\mathbb{T} = [0,T]$ and extend the observed network to all $t \in \mathbb{T}$ by defining $G_t$ to be piecewise constant on the intervals $[t_i, t_{i+1})$, that is, $G_t = G_{t_i}$ for $t \in [t_i, t_{i+1})$.

The other set of tools we need to write the problem formulation and solution are matrix functions and some related properties.

\begin{definition}[{Matrix Function~\cite[Definition~1.2]{HighamBook}}]
Let \( f : \mathbb{C} \to \mathbb{C} \) be a scalar function that is defined and analytic on the spectrum of an \( n \times n \) matrix \( A \in \mathbb{C}^{n \times n} \). The \textit{matrix function} \( f(A) \) can be defined in terms of the Jordan canonical form of \( A = X J X^{-1} \). The matrix \( J \) has a block-diagonal form with Jordan blocks \( J_k \) corresponding to each eigenvalue \( \lambda_k \) of \( A \):
\[
J = \text{diag}(J_1, J_2, \dots, J_m),
\]
where each Jordan block \( J_k \) has the form
\[
J_k = \begin{bmatrix}
\lambda_k & 1 & 0 & \cdots & 0 \\
0 & \lambda_k & 1 & \cdots & 0 \\
\vdots & \vdots & \ddots & \ddots & \vdots \\
0 & 0 & \cdots & \lambda_k & 1 \\
0 & 0 & \cdots & 0 & \lambda_k
\end{bmatrix}.
\]
The matrix function \( f(A) \) is then defined by applying \( f \) to each Jordan block~\( J_k \):
\[
f(A) = X \begin{bmatrix} f(J_1) & 0 & \cdots & 0 \\ 0 & f(J_2) & \cdots & 0 \\ \vdots & \vdots & \ddots & \vdots \\ 0 & 0 & \cdots & f(J_m) \end{bmatrix} X^{-1},
\]
where \( f(J_k) \) is defined by applying \( f \) to each entry along the superdiagonal in \( J_k \) using the formula
\[
f(J_k) = \begin{bmatrix} f(\lambda_k) & f'(\lambda_k) & \frac{f^{(2)}(\lambda_k)}{2!} & \cdots & \frac{f^{(n_k-1)}(\lambda_k)}{(n_k-1)!} \\ 0 & f(\lambda_k) & f'(\lambda_k) & \cdots & \frac{f^{(n_k-2)}(\lambda_k)}{(n_k-2)!} \\ \vdots & \vdots & \ddots & \ddots & \vdots \\ 0 & 0 & \cdots & f(\lambda_k) & f'(\lambda_k) \\ 0 & 0 & \cdots & 0 & f(\lambda_k) \end{bmatrix}.
\]
Here, \( n_k \) is the size of the Jordan block \( J_k \), and \( f^{(j)}(\lambda_k) \) denotes the \( j \)-th derivative of \( f \) evaluated at \( \lambda_k \).
\end{definition}

\begin{definition}[Fréchet Derivative of a Matrix Function]\label{def:frechet}
Let \( f : \mathbb{R}^{n \times n} \to \mathbb{R}^{n \times n} \) be a matrix function. The \textit{Fréchet derivative} of \( f \) at \( A \in \mathbb{R}^{n \times n} \) in the direction $E$ is a linear map \( \mathrm{L}_f : \mathbb{R}^{n \times n} \to \mathbb{R}^{n \times n} \) that approximates the change in \( f \) due to a small perturbation \( E \in \mathbb{R}^{n \times n} \) in \( A \). Formally, \( \mathrm{L}_f \) is the Fréchet derivative of \( f \) at \( A \) in the direction $E$ if
\[
f(A + E) = f(A) + \mathrm{L}_f(A,E) + o(\|E\|),
\]
where \( o(\|E\|) \) denotes a term that vanishes faster than \( \|E\| \) as \( \|E\| \to 0 \).
\end{definition}

This can be computed for a $(2n-1)$-times differentiable function $f$  in an open set containing the eigenvalues of $A$ as the matrix function of the \(2 \times 2\) block matrix~\cite[Theorem~3.8]{HighamBook}
\begin{equation}\label{eq:explicit-frechet}
    \mathrm{L}_f(A, E) = \left[ \begin{matrix} I & O \end{matrix} \right] f\left( \begin{bmatrix} A & E \\ O & A \end{bmatrix} \right) \left[ \begin{matrix} O \\ I \end{matrix} \right],
\end{equation}
where \( I \) and \(O\) are the \( n \times n \) identity and zero matrices, respectively. For some particular functions $f$ it is possible to find other expressions of $\mathrm{L}_f(A,E)$, e.g., in the case of the logarithm~\cite{MR3080997}.

Finally, we recall the definition of Katz's centrality index~\cite{Katz_1953} for a static network, which we will need for using the collection of Katz centralities computed for the different time steps $t \in \mathbb{T}$.

\begin{definition}\label{def:katz-centrality}
Let $G = (V, E, w)$ be a static weighted network with $n$ vertices and adjacency matrix $A \in \mathbb{R}^{n \times n}$. The \emph{Katz centrality} of the nodes in $G$ is defined as the vector $\boldsymbol{\mu} \in \mathbb{R}^n$ that satisfies the relationship: $\boldsymbol{\mu} = a A \boldsymbol{\mu} + \mathbf{1}$, where $a > 0$ is a scalar attenuation factor, and $\mathbf{1} \in \mathbb{R}^n$ is the column vector of all ones. Provided that $a < \frac{1}{\rho(A)}$, where $\rho(A)$ denotes the spectral radius of $A$, the Katz centrality vector is well-defined and can be computed via the matrix resolvent as $
\boldsymbol{\mu} = (I - a A)^{-1} \mathbf{1}$,
where $I$ is the $n \times n$ identity matrix.
\end{definition}

\section{The dynamical system formulation of the network centrality}\label{sec:dynamical_system}

With the goal of quantifying the ability of node $v_i$ to communicate with node $v_j$ over time in a dynamic network, in \cite{HighamNet}  a dynamical system formulation is introduced for defining a \emph{dynamic communicability matrix} for the network. The foundation of this approach is the time-dependent adjacency matrix $A(t)$ (Definition~\ref{def:time-adjacency}), where edges appear, disappear and, generally, modify their weights as the network evolves. The dynamic communicability matrix $S(t)$ is constructed by counting dynamic walks that involve edges present at time $t$ or earlier, while also accounting for the temporal decay of information relevance.

To update $S(t)$ over small time intervals $\delta t$,  the following formula is introduced in~\cite{HighamNet}:
\begin{equation}\label{eq:discrete-formulation}
   S(t + \delta t) = (I + e^{-b \delta t} S(t)) (I - a A(t + \delta t))^{-\delta t} - I, 
\end{equation}
where $S(0) = 0$, $I$ is the identity matrix, $a > 0$ is the edge attenuation parameter that controls the weight assigned to walks based on the number of edges, and $b > 0$ is the temporal downweighting parameter that governs how older walks are penalized. This formulation allows the matrix $S(t)$ to count walks between nodes while diminishing the importance of longer and older walks.

\begin{remark}\label{rem:matrix_log_existence}
Observe that the parameter $a>0$ can be chosen such that $\rho(A(t+\delta t)) < 1/a$ for all $0 \leq \delta t < \delta^*$. It then follows that $I - aA(t+\delta t)$ is a nonsingular $M$-matrix; see, for example,~\cite[\S 6, Definition 1.2]{BermanPlemmons}. Consequently, $(I - aA(t+\delta t))^{\delta t}$ is also an $M$-matrix~\cite{MR700883}, and therefore its inverse satisfies $(I - aA(t+\delta t))^{-\delta t} \geq 0$. Since $S(0) = 0$ and the update formula preserves nonnegativity, an induction argument shows that $S(t) \geq 0$ for all $t \geq 0$, and in particular $S(t+\delta t) \geq 0$.
\end{remark}

Taking the limit as $\delta t \to 0$ in~\eqref{eq:discrete-formulation}, the following ordinary differential equation (ODE) for the communicability matrix is obtained:
\begin{equation}\label{eq:ODE-formulation}
\dot{C}(t) = -b (C(t) - I) - C(t) \log (I - a A(t)),
\end{equation}
where $C(t) = I + S(t)$ and $\log$ denotes the matrix logarithm. In this ODE, the term $-b (C(t) - I)$ ensures that older walks are downweighted over time, while the term $-C(t) \log (I - a A(t))$ accounts for the dynamic walks utilizing edges present at time $t$.
\begin{remark}\label{rem:log-uniqueness}
    We recall that a logarithm of $M \in \mathbb{C}^{n \times n}$ is any matrix $X$ such that $e^X = M$. Any nonsingular $M$ has infinitely many logarithms. To enforce uniqueness we need to require that $M$ has no eigenvalues on $\mathbb{R}^-$ and choose the principal determination of the multivalued $\log_{\mathbb{C}}$ function~\cite[Theorem 1.31, p.20]{HighamBook}, i.e., the unique logarithm whose spectrum lies in the strip $\{ z \in \mathbb{C} \,:\, -\pi < \Im(z) < \pi \}$. For the case at hand, 
    $M = I - a A(t)$.
    Hence, the natural requirement on the $a$ parameter in~\eqref{eq:ODE-formulation} is %
    $\displaystyle a < \inf_{t} \nicefrac{1}{\rho(A(t))}$, 
    for $\rho(\cdot)$ the spectral radius. This condition ensures that $I-aA(t)$ is a nonsingular M-matrix, hence its eigenvalues lie in the open right half plane. This condition needs to be fixed at the outset before the whole computation.
\end{remark}

The matrix $C(t)$ encodes the communication capability at time $t$ between node $i$ and node $j$ through dynamic walks in the network, applying downweighting to both the length of walks and the age of information.

To derive a more efficient collective measure of the importance of the nodes in the time-evolving network from the communicability matrix $C(t)$, as in \cite{HighamNet}, we define the \textit{dynamic broadcast centrality} and \textit{dynamic receive centrality} vectors as:
\begin{equation}\label{eq:broadcast_and_receiving}
\mathbf{b}(t) = C(t) \mathbf{1} \quad \text{and} \quad \mathbf{r}(t) = C(t)^\top \mathbf{1},    
\end{equation}
where $\mathbf{1}$ is the vector of ones. These measures capture the propensity of each node to broadcast and receive information, respectively.

\subsection{Extensions and generalizations}\label{sec:extension_and_generalizations}
All these formulations can be generalized to the case 
\begin{equation}\label{eq:generalized}
    \dot{C}(t) = - b( C(t) - I ) + C(t) \log(H(A(t))),  \qquad C(0) = I,
\end{equation}
where $H(A(t))$ is a given matrix function, i.e., we consider a given combination of walk lengths in the model. A choice that greatly simplifies the construction is $H(x) = \exp(a x)$. With this choice the entire formulation of the problem is reduced to the linear case
\begin{equation}\label{eq:simplified}
    \dot{C}(t) = - b( C(t) - I ) + a C(t) A(t),  \qquad C(0) = I,
\end{equation}
if $\log(\exp(a A(t) )) = a A(t)$, which is true when $|\Im(\lambda_i( a A(t)))| < \pi$, for all $i = 1,\ldots,n$ and $t \in \mathbb{T}$, or, for any consistent matrix norm $\|a A(t)\| < \pi$ for all $t \in \mathbb{T}$; see \cite[Exercise 1.39, p.32]{HighamBook}. In the case in which $A(t) = A(t)^\top$, i.e., the network is undirected this always holds. Otherwise, we need to select $a$ in such a way that $|\Im(\lambda_i(a A(t)))| < \pi$ for all $t \in \mathbb{T}$, $i=1,\ldots,n$.

\section{Steering the dynamic centralities to prescribed values}\label{sec:the-steering-problem}

Suppose we now have access to a network that evolves over time, $G_t$, and that we can selectively intervene at each time $t$ on the weights of all or some of the edges. Our objective is to steer the network's communicability, represented in matrix form, or its broadcast and reception centrality indices, towards a desired state at a given time $T$. We can formulate these problems as optimal control problems for the ODEs~\eqref{eq:ODE-formulation} or~\eqref{eq:simplified}, or for the derived broadcast and receiving centralities~\eqref{eq:broadcast_and_receiving}. 

For the communicability matrix, given a desired state $C^* \in \mathbb{R}^{n\times n}$ at time $T$ and a set $\mathcal{P} \subseteq V \times V $,
we can express the problem as:
\begin{equation}\label{eq:matrix-formulation}
\begin{array}{rl}
    \displaystyle \min_{U} & \displaystyle \mathrm{J}(C,U) =  \frac{1}{2}\| C(T) - C^* \|_F^2 + \frac{\alpha}{2}\int_{0}^{T} \| U(t) \|_F^2 \, \mathrm{d}t, \qquad \alpha > 0\\
    \text{s.t.} &  \begin{cases}\dot{C}(t) = -b (C(t) - I) - C(t) \log \left(I - a (A(t) + U(t)) \right), \\
    C(0) = I, 
    \end{cases}
    \\
    & A(t) + U(t) \geq 0, \\
    & \rho(A(t) + U(t)) < \nicefrac{1}{a},  \\
    & [U(t)]_{i,j} = 0 \text{ if } (i,j) \not\in \mathcal{P},
\end{array}  
\end{equation}
where we require that the control $U(t)$ be a matrix of the same size as $A(t)$, maintaining a sparsity pattern contained in the predetermined set $\mathcal{P}$, that the weights of the perturbed network remain nonnegative, and that the uniqueness condition for the matrix logarithm remains satisfied along the solution trajectory as described in Remark~\ref{rem:log-uniqueness}. The $\alpha \in \mathbb{R}_{> 0}$ value in the cost functional is a Tikhonov regularization parameter of the control.

A similar formulation can be applied to the dynamic receive case, where we selected a desired receive centrality vector $\mathbf{r}^* \geq \boldsymbol{1} \in \mathbb{R}^n$:
\begin{equation}\label{eq:dynamic-receive-original}
\begin{array}{rl}
    \displaystyle \min_{U} & \displaystyle \mathrm{J}(\mathbf{r},U) =  \frac{1}{2}\| \mathbf{r}(T) - \mathbf{r}^* \|_{2}^2 + \frac{\alpha}{2}\int_{0}^{T} \| U(t) \|_F^2 \, \mathrm{d}t, \qquad \alpha > 0 \\
    \text{s.t.} &  \begin{cases}\dot{\mathbf{r}}(t) = -b (\mathbf{r}(t) - \mathbf{1}) - \log \left(I - a (A(t) + U(t)) \right)^\top \mathbf{r}(t), \\
    \mathbf{r}(0) = \mathbf{1}, \end{cases}\\
    & A(t) + U(t) \geq 0, \\
    & \rho(A(t) + U(t)) < \nicefrac{1}{a}, \\
    & [U(t)]_{i,j} = 0 \text{ if } (i,j) \not\in \mathcal{P}.
\end{array}  
\end{equation}
Observe that the broadcast vector $\mathbf{b}(t)$ tracks the information that has flowed out of each node, thus it requires access to the entries of the $C(t)$ matrix. Consequently, the receive centrality is fundamentally simpler to solve than the broadcast centrality.

If we consider the extended framework discussed in~\cite{HighamNet} with a generic matrix function $H(A(t))$ and select $H(A(t))=\exp(a A(t))$, we may employ the simplified linear model~\eqref{eq:simplified}, so as to lead us to the solution of the control problem
\begin{equation}\label{eq:dynamic-receive-simplified-original}
    \begin{array}{rl}
     \displaystyle \min_{U} & \displaystyle \mathrm{J}(\mathbf{r},U) = \frac{1}{2}\| \mathbf{r}(T) - \mathbf{r}^*\|_2^2 + \frac{\alpha}{2} \int_{0}^{T} \| U(t)\|_F^2 \mathrm{d}t, \qquad \alpha > 0 \\
     \text{ s.t. }& \begin{cases}\dot{\mathbf{r}}(t) = - b (\mathbf{r}(t) - \mathbf{1}) + a (A(t) + U(t))^\top \mathbf{r}(t),\\
     \mathbf{r}(0) = \mathbf{1}, \end{cases}\\
     & A(t) + U(t) \geq 0, \\
     & |\Im(\lambda_i(A(t) + U(t)))| < \nicefrac{\pi}{a}, \quad \forall\, i = 1,\ldots,n \\
     & [U(t)]_{i,j} = 0 \text{ if } (i,j) \not\in \mathcal{P}.
\end{array} 
\end{equation}
Observe that although we no longer compute any matrix logarithm in the state equation of~\eqref{eq:dynamic-receive-simplified-original}, to have the transformation between~\eqref{eq:generalized} and~\eqref{eq:simplified} we need the condition $\log(\exp(A)) = A$, hence the condition on the imaginary part of the spectrum. Observe also that if both $G_t$ is undirected and  that $U(t) = U(t)^\top$ for all $t \in \mathbb{T}$, then such a condition can be removed; see the discussion in Section~\ref{sec:extension_and_generalizations}. In the case of directed graphs, a sufficient condition for the constraint on the imaginary part of the eigenvalues can be stated in terms of the outdegree, as in Proposition~\ref{prop:relaxed_constraint_linear}.

\subsection{Relaxing the spectral constraints}
The constraint on the spectral radius of $A(t)+U(t)$ in~\eqref{eq:matrix-formulation} is nonlinear and nonconvex, hence its use complicates any optimization procedure one can use to solve the control problem. To obtain a more efficient procedure with a linear constraint, we consider the following relaxation which makes use of the collection of Katz centrality measures (Definition~\ref{def:katz-centrality}) for all $t \in \mathbb{T}$, and establishes first the existence of a choice of the $a$ parameter which makes Katz centrality well defined for a general time evolving graph (Definition~\ref{def:time-evolving-graph}).
\begin{lemma}\label{lem:katz_exists}
Let $A(t)$ be the adjacency matrix at time $t$. If
$
0< a < \inf_{t} \nicefrac{1}{\rho(A(t))}$,
then the Katz centrality at time $t$
$
\boldsymbol{\mu}(t) = (I - a A(t))^{-1}\mathbf{1}
$
is well defined.
\end{lemma}

\begin{proof}
For the Katz centrality to exist, it is needed that the resolvent matrix $I - aA(t)$ is invertible. If 
$
0< a < \inf_t \nicefrac{1}{\rho(A(t))}, 
$
then 
$\rho(aA(t)) < 1$ for every $t$, which implies that $I - aA(t)$ is invertible; consequently the vectors $\boldsymbol{\mu}(t)$ are well-defined.
\end{proof}

\begin{proposition}\label{prop:relaxed_constraint}
Let $a$ and $A(t)$ be as in Lemma~\ref{lem:katz_exists}, and let $\boldsymbol{\mu}(t)$ be the Katz centrality associated with the matrix $A(t)$ and parameter $a$. If $U(t)$ is such that $(I - a(A(t)+U(t)))\boldsymbol{\mu}(t) \geq \mathbf{0}$ and $A(t)+U(t) \geq 0$, then
$\rho(A(t)+U(t)) \le \nicefrac{1}{a}$.
\end{proposition}

\begin{proof}
From the assumption we have $(I - a(A(t)+U(t)))\boldsymbol{\mu}(t) \geq \mathbf{0}$, which is equivalent to $a(A(t)+U(t))\boldsymbol{\mu}(t)\leq \boldsymbol{\mu}(t)$.
Since $\boldsymbol{\mu}(t) \ge \mathbf{1} > \mathbf{0}$, we define $D_{\boldsymbol{\mu}(t)} = \operatorname{diag}( \boldsymbol{\mu}(t))$ and find that
\[
a\;D_{\boldsymbol{\mu}(t)}^{-1}
(A(t)+U(t))\;
D_{\boldsymbol{\mu}(t)} \mathbf{1} \le \mathbf{1}.
\]
Since the matrix $A(t)+U(t)$ is nonnegative, the above inequality shows that 
\[a\;\|D_{\boldsymbol{\mu}(t)}^{-1}
(A(t)+U(t))\;
D_{\boldsymbol{\mu}(t)}\|_\infty \le 1,\] hence
$
a\rho(D_{\boldsymbol{\mu}(t)}^{-1}
(A(t)+U(t))\;
D_{\boldsymbol{\mu}(t)})=a \rho(A(t)+U(t))\le 1$.
\end{proof}

According to
Proposition~\ref{prop:relaxed_constraint},
the linear inequality
$
(I - a(A(t)+U(t)))\boldsymbol{\mu}(t) \geq \mathbf{0}$,  
which is easier  to handle computationally than the nonlinear  spectral radius constraint,
guarantees that
$\rho((A(t)+U(t))\le \nicefrac{1}{a}$.
To provide a sufficient condition guaranteeing that the strict inequality $\rho(A(t)+U(t)) < \nicefrac{1}{a}$ is satisfied, we then need to just add a safety guarantee of a small $\frac{1}{a} \gg \epsilon>0$, asking that $U(t) \boldsymbol{\mu}(t) \leq \left( \nicefrac{1}{a} - \epsilon \right)\boldsymbol{1}$, where we used the fact that $(I-aA(t))\boldsymbol{\mu}(t) = \mathbf{1}$.

We need a relaxation also for the constraint of the imaginary part of the spectrum in problem~\eqref{eq:dynamic-receive-simplified-original} in the case of directed graphs. We can employ the observation on the consistent matrix norms made in Section~\ref{sec:extension_and_generalizations}.
\begin{proposition}\label{prop:relaxed_constraint_linear}
Let $a > 0$ and $A(t)\ge 0$. If $U(t)$ is such that $A(t) + U(t) \geq 0 $ and
\begin{equation}\label{eq:linear_constraint_for_im_part}
     (A(t) + U(t)) \mathbf{1} \leq \frac{\pi}{a} \mathbf{1}, 
\end{equation}
then $\left| \Im( \lambda_i(a (A(t)+U(t)) ) \right|\le \pi$.
\end{proposition}

\begin{proof}
If $A(t) + U(t) \geq 0$, then the condition~\eqref{eq:linear_constraint_for_im_part} implies that $\| a (A(t) + U(t)) \|_\infty \leq \pi$, hence that $\left|\Im( \lambda_i(a (A(t)+U(t)) ) \right| \le \pi$, since $\|\cdot\|_\infty$ is a compatible norm.
\end{proof}

As for the other relaxation, we need to add a safety guarantee of a small $\frac{\pi}{a} \gg \varepsilon > 0$, asking that
\[
(A(t) + U(t)) \mathbf{1} \leq \left( \frac{\pi}{a}  - \varepsilon \right) \mathbf{1},
\]
observe also that in network terms, this amounts to a constraint on the out-degree of the network. An immediate corollary of Proposition~\ref{prop:relaxed_constraint_linear} is then obtained by changing from out-degree to in-degree and from the compatible norm $\|\cdot\|_\infty$ to the compatible norm $\|\cdot\|_1$.
\begin{corollary}
Let $a > 0$ and $A(t)\ge 0$. If $U(t)$ is such that $A(t) + U(t) \geq 0 $ and
     $\mathbf{1}^\top (A(t) + U(t))  \leq \nicefrac{\pi}{a} \mathbf{1}^\top$, 
then $\left| \Im( \lambda_i(a (A(t)+U(t)) ) \right| \le \pi$.
\end{corollary}

We can therefore rewrite the two problems~\eqref{eq:dynamic-receive-original}, \eqref{eq:dynamic-receive-simplified-original} using the simplified bounds, i.e., for the case of~\eqref{eq:dynamic-receive-original}, we rewrite it as
\begin{equation}\label{eq:dynamic-receive}
\begin{array}{rl}
    \displaystyle \min_{U} & \displaystyle \mathrm{J}(\mathbf{r},U) =  \frac{1}{2}\| \mathbf{r}(T) - \mathbf{r}^* \|_{2}^2 + \frac{\alpha}{2}\int_{0}^{T} \| U(t) \|_F^2 \, \mathrm{d}t, \qquad \alpha > 0, \; \epsilon > 0 \\
    \text{s.t.} &  \begin{cases}\dot{\mathbf{r}}(t) = -b (\mathbf{r}(t) - \mathbf{1}) - \log \left(I - a (A(t) + U(t)) \right)^\top \mathbf{r}(t), \\
    \mathbf{r}(0) = \mathbf{1}, \end{cases}\\
    & A(t) + U(t) \geq 0, \\
   &  U(t) \boldsymbol{\mu}(t) \leq \left( \nicefrac{1}{a} - \epsilon \right)\boldsymbol{1}, \\
    & [U(t)]_{i,j} = 0 \text{ if } (i,j) \not\in \mathcal{P},
\end{array}  
\end{equation}
while we simplify the constraints in \eqref{eq:dynamic-receive-simplified-original} for a directed graph as
\begin{equation}\label{eq:dynamic-receive-simplified}
    \begin{array}{rl}
     \displaystyle \min_{U} & \displaystyle \mathrm{J}(\mathbf{r},U) = \frac{1}{2}\| \mathbf{r}(T) - \mathbf{r}^*\|_2^2 + \frac{\alpha}{2} \int_{0}^{T} \| U(t)\|_F^2 \mathrm{d}t, \qquad \alpha > 0, \; \epsilon > 0 \\
     \text{ s.t. }& \begin{cases}\dot{\mathbf{r}}(t) = - b (\mathbf{r}(t) - \mathbf{1}) + a (A(t) + U(t))^\top \mathbf{r}(t),\\
     \mathbf{r}(0) = \mathbf{1}, \end{cases}\\
     & A(t) + U(t) \geq 0, \\
     &    (A(t) + U(t)) \mathbf{1} \leq \left( \frac{\pi}{a} -\epsilon\right) \mathbf{1}, \\
     & [U(t)]_{i,j} = 0 \text{ if } (i,j) \not\in \mathcal{P},
\end{array} 
\end{equation}
while for an undirected graph, where $A(t) = A^\top(t)$ for all times, we have
\begin{equation}\label{eq:dynamic-receive-simplified-undirected}
    \begin{array}{rl}
     \displaystyle \min_{U} & \displaystyle \mathrm{J}(\mathbf{r},U) = \frac{1}{2}\| \mathbf{r}(T) - \mathbf{r}^*\|_2^2 + \frac{\alpha}{2} \int_{0}^{T} \| U(t)\|_F^2 \mathrm{d}t, \qquad \alpha > 0 \\
     \text{ s.t. }& \begin{cases}\dot{\mathbf{r}}(t) = - b (\mathbf{r}(t) - \mathbf{1}) + a (A(t) + U(t))^\top \mathbf{r}(t),\\
     \mathbf{r}(0) = \mathbf{1}, \end{cases}\\
     & A(t) + U(t) \geq 0, \\
     & [U(t)]_{i,j} = 0 \text{ if } (i,j) \not\in \mathcal{P}, \\
     & U(t) = U^\top(t).
\end{array} 
\end{equation}

We can prove that the admissible set of controls in the three cases are not empty under reasonable choices of the modification pattern $\mathcal{P}$.

\begin{proposition}\label{pro:nonemptyness_logarithm}
The admissible set for problem~\eqref{eq:dynamic-receive}
\begin{equation}\label{eq:admset_original}
\mathcal{U}_{ad}(t) = \left\{ U(t) \in \mathbb{R}^{n \times n} \;\middle|\;   \begin{aligned}  & A(t) + U(t) \geq 0, \\
   &  U(t) \boldsymbol{\mu}(t) \leq \left( \nicefrac{1}{a} - \epsilon \right)\boldsymbol{1}, \\
    & [U(t)]_{i,j} = 0 \text{ if } (i,j) \not\in \mathcal{P} \end{aligned} \right\rbrace,
\end{equation}
is non trivial for $\emptyset\neq\mathcal{P}\subseteq V \times V$, i.e., $\mathcal{U}_{ad}(t) \setminus \{O\} \neq \emptyset$ for $O$ the zero matrix.
\end{proposition}

\begin{proof}
    Let $H$ be the $n \times n$ binary indicator matrix representing the sparsity pattern $\mathcal{P}$, such that $h_{i,j} = 1$ if $(i,j) \in \mathcal{P}$ and $h_{i,j} = 0$ otherwise. We seek to construct a matrix $U(t) \in \mathcal{U}_{ad}(t)$ of the form $U(t) = D H$, where $D = \operatorname{diag}(x_1, \dots, x_n)$ with $x_i \geq 0$ for all $i \in \{1, \dots, n\}$. By construction, $[U(t)]_{i,j} = x_i h_{i,j} = 0$ if $(i,j) \notin \mathcal{P}$, which satisfies the sparsity constraint.
    Since the matrix $A(t)$ has non-negative entries, the condition $A(t) + U(t) \geq 0$ is rewritten as $a_{i,j}(t) + x_i h_{i,j} \geq 0$ for all $i,j$. Since we restrict $x_i \geq 0$ and $h_{i,j} \geq 0$, the term $x_i h_{i,j}$ is non-negative, meaning the inequality is satisfied.

    Next, we consider the constraint $U(t) \boldsymbol{\mu}(t) \leq \left( \nicefrac{1}{a} - \epsilon \right)\boldsymbol{1}$. For the $i$-th row, this gives
    \[
    x_i \sum_{j=1}^n h_{i,j} \mu_j(t) \leq \nicefrac{1}{a} - \epsilon.
    \]
    Let $S_i(t) = \sum_{j=1}^n h_{i,j} \mu_j(t)$. Provided that $\nicefrac{1}{a} - \epsilon  > 0$, this inequality is satisfied for each row $i$ if either
    $S_i(t) = 0$---for which any $x_i \geq 0$ is a valid choice---or, if $S_i(t) > 0$---then, the condition is satisfied for any $x_i$ chosen such that $0 \leq x_i \leq \nicefrac{\left(\nicefrac{1}{a} - \epsilon\right)}{S_i(t)}$.
    To prove that $\mathcal{U}_{ad}(t)$ is non-trivial ($\mathcal{U}_{ad}(t) \setminus \{O\} \neq \emptyset$), we must show there exists a valid $U(t)$ other than the zero matrix. Because $\mathcal{P} \neq \emptyset$, the matrix $H$ has at least one non-zero entry $h_{i,j}$, and we can choose $0<x_i\le \nicefrac{\left(\nicefrac{1}{a} - \epsilon\right)}{S_i(t)}$, so that $U(t)\ne O$. Therefore, the admissible set is non-trivial.
\end{proof}

We introduce the admissible control set for the linear model for the general case of directed graphs~\eqref{eq:dynamic-receive-simplified}:
\begin{equation}
    \label{eq:adm_set_directed}
    \mathcal{U}_{ad}(t) = \left\{ U(t) \in \mathbb{R}^{n \times n} \;\middle|\; \begin{aligned} &  U(t) \ge -A(t), \\ & U(t) \mathbf{1} \leq \left( \frac{\pi}{a}  - \varepsilon \right) \mathbf{1} - A(t)\mathbf{1}, \\  & [U(t)]_{i,j} = 0 \text{ if } (i,j) \notin \mathcal{P} \end{aligned}  \right\},
\end{equation}
while for undirected graphs case~\eqref{eq:dynamic-receive-simplified-undirected} we have
\begin{equation}
    \label{eq:adm_set}
    \mathcal{U}_{ad}(t) = \left\{ U(t) \in \mathbb{R}^{n \times n} \;\middle|\; \begin{aligned}  &  U(t) \ge -A(t), \\ & U(t) = U^\top(t) , \\ & [U(t)]_{i,j} = 0 \text{ if } (i,j) \notin \mathcal{P} \end{aligned}   \right\},
\end{equation}
which are not empty under a stricter assumption on the pattern of the admissible modifications.
\begin{proposition}\label{pro:non_emptyness_linear}
The sets~\eqref{eq:adm_set_directed} and~\eqref{eq:adm_set} are non trivial for $\mathcal{P} = \cup_{t \in \mathbb{T}} E_t$, i.e., $\mathcal{U}_{ad}(t) \setminus \{O\} \neq \emptyset$ for $O$ the zero matrix.

\end{proposition}

\begin{proof}
We prove first the  case~\eqref{eq:adm_set_directed}. Let us define $D = \operatorname{diag}(x_1,\ldots,x_n)$ and search for a $U(t) \in \mathcal{U}_{ad}(t)$ of the form $U(t) = DA(t)$. The conditions~\eqref{eq:adm_set_directed} are then equivalent to
\[
\begin{split}
    U(t) + A(t) \geq 0 \,\Leftrightarrow\,& (I + D)A(t) \geq 0 \,\Leftrightarrow \, (1 + x_i) a_{i,j}(t) \geq 0 \\
 U(t) \mathbf{1} \leq \left( \frac{\pi}{a}  - \varepsilon \right) \mathbf{1} - A(t) \mathbf{1}  \,\Leftrightarrow\,& x_i d_i \leq \frac{\pi}{a}-\epsilon - d_i, \text{ with } d_i = (A(t)\mathbf{1})_i.   
\end{split}
\]
If $d_i=0$, i.e, $a_{i,j}(t)=0$ for any $j$, then for any choice of $x_i$ the above inequalities are satisfied.
If $d_i> 0$ then we can choose $x_i\ge -1$ and $x_i \leq \frac{\frac{\pi}{a}-\epsilon}{d_i} - 1$; since $\nicefrac{\left(\frac{\pi}{a}-\epsilon\right)}{d_i} > 0$, there are infinite values of $x_i$ satisfying both the inequalities.
By construction, $[U(t)]_{i,j} = 0$ if  $(i,j) \not\in \mathcal{P}$.
For~\eqref{eq:adm_set} we look for $U(t) = D A(t) D$, $\operatorname{diag}(x_1,\ldots,x_n)$, and we need to prove
\[
\begin{split}
    U(t) + A(t) \geq 0 \,\Leftrightarrow\,&   D A(t) D + A(t) \geq 0 \,\Leftrightarrow \, x_i a_{i,j}(t) x_j \geq - a_{i,j}(t), %
\end{split}
\]
if $a_{i,j}(t) = 0$, then the inequality is verified for every choice of $D$. If $a_{i,j}(t) > 0$, the condition becomes $x_i x_j \geq -1$, which is satisfied for every $(x_1,\ldots,x_n)$ with constant sign. Again, by construction, $[U(t)]_{i,j} = 0$ if  $(i,j) \not\in \mathcal{P}$.
\end{proof}

Propositions~\ref{pro:nonemptyness_logarithm} and~\ref{pro:non_emptyness_linear} suggest that a reasonable choice for the modification pattern $\mathcal{P}$ is therefore the collection of the patterns of the $A(t)$ for all $t \in \mathbb{T}$: from a modeling perspective, this choice means that we want to modify edges that existed at some point in the network's temporal evolution.

\subsection{Existence of the optimal control}

To guarantee that the optimal control problems formulated in the previous sections are well-posed, we must establish the existence of an optimal control $U^*(t)$ and its corresponding state trajectory. We state the existence result focusing on the simplified linear model~\eqref{eq:dynamic-receive-simplified}, utilizing the classic Filippov-Cesari existence theorem for Bolza problems~\cite[Theorem 9.3i]{cesari1983optimization}.

\begin{theorem}
For the dynamic receive optimal control problems defined in
\eqref{eq:dynamic-receive-simplified} and
\eqref{eq:dynamic-receive-simplified-undirected}, there exists an optimal control
$U^*(t) \in \mathcal{U}_{\mathrm{ad}}(t)$, where
$\mathcal{U}_{\mathrm{ad}}(t)$ denotes the admissible control set defined in
\eqref{eq:adm_set_directed} and \eqref{eq:adm_set}, respectively, together with an absolutely continuous optimal state trajectory
$\mathbf{r}^*(t)$ for $t \in [0,T]$.
\end{theorem}

\begin{proof}
We verify the required conditions of the Filippov-Cesari theorem (\cite[Theorem 9.3i]{cesari1983optimization}) step by step. For every admissible control
\(U(t)\), we define
\[
    f(t,\mathbf r,U)
    :=
    - b(\mathbf r-\mathbf 1)
    + a(A(t)+U)^\top \mathbf r .
\]
Thus the controlled dynamics can be written compactly as
\[
    \dot{\mathbf r}(t)=f(t,\mathbf r(t),U(t)),
    \qquad \mathbf r(0)=\mathbf 1 .
\]

Under the linear constraints of~\eqref{eq:adm_set}, the control is bounded from below by the constraint $U(t) \ge -A(t)$, but remains unbounded from above. Thus, $\mathcal{U}_{ad}(t)$ is a closed, convex, but unbounded set. However, the running cost $L(U) = \frac{\alpha}{2}\|U\|_F^2$ is strictly convex and coercive. The coercivity guarantees that the optimal control is bounded, meaning there exists a sufficiently large constant $R > 0$ such that the optimal control must satisfy $\|U(t)\|_F \le R$ almost everywhere. Without loss of generality, we can intersect our original control set with this closed ball to define an effective control set $\widetilde{\mathcal{U}}_{ad}(t) = \mathcal{U}_{ad}(t) \cap \{ U \mid \|U\|_F \le R \}$. This effective set is both closed and bounded in a finite-dimensional space, satisfying the compactness requirement. Under the constraint~\eqref{eq:adm_set_directed}, we can take instead $\widetilde{\mathcal{U}}_{ad}(t) = \mathcal{U}_{ad}(t)$ which is already closed and bounded in a finite-dimensional space.

Second, we establish the existence and uniqueness of the state trajectories. For any measurable control $U(t) \in \widetilde{\mathcal{U}}_{ad}(t)$, $f$ is continuous in $t$ and uniformly Lipschitz continuous with respect to the state $\mathbf{r}$. This holds because the spectral norm $\|A(t)+U(t)\|_2$ is uniformly bounded over the compact domain $\widetilde{\mathcal{U}}_{ad}(t)$. By the Carath\'{e}odory existence theorem, there exists a unique, absolutely continuous trajectory $\mathbf{r}(t)$ on $[0,T]$. 

Third, we analyze the extended velocity set $N(\mathbf{r}, t)$, which characterizes the reachable directions penalized by the running cost:
\begin{equation*}
    N(\mathbf{r}, t) = \left\{ (f(t, \mathbf{r}, U), \gamma) \;\middle|\; U \in \widetilde{\mathcal{U}}_{ad}(t), \gamma \ge \frac{\alpha}{2}\|U\|_F^2 \right\}.
\end{equation*}
We recall that the state dynamics $f(t, \mathbf{r}, U)$ are affine with respect to the control matrix $U$ and the running cost $L(U) = \frac{\alpha}{2}\|U\|_F^2$ is strictly convex with respect to $U$. Because the mapping $U \mapsto f$ is affine, the mapping $U \mapsto L$ is convex, and the domain $\widetilde{\mathcal{U}}_{ad}(t)$ is a convex set, it follows that the extended velocity set $N(\mathbf{r}, t)$ is convex.

Since the effective admissible control set is compact, the dynamics are well-posed and bounded, and the extended velocity set is convex, all requirements of the Filippov-Cesari theorem are satisfied. Therefore, an absolute minimum for the functional $J(\mathbf{r},U)$ within the class of admissible pairs is guaranteed to exist.
\end{proof}

\begin{remark}
    For the nonlinear formulations~\eqref{eq:matrix-formulation} and~\eqref{eq:dynamic-receive}, the dynamics involve the nonlinear term $-\log(I - a(A(t)+U(t)))$. The non-affine nature of the logarithm implies the mapping $U \mapsto f(t, C, U)$ is nonlinear. Consequently, the extended velocity set $N(C, t)$ may lose convexity with respect to $U \in \mathcal{U}_{ad}(t)$.  Consequently, standard existence results like the Filippov-Cesari theorem cannot be directly applied. Establishing the existence of an optimal control for this log-based formulation presents significant mathematical challenges and falls beyond the scope of this paper. Nonetheless, Proposition~\ref{pro:nonemptyness_logarithm} is sufficient to ensure that at least the chosen set of admissible controls is nonempty for a reasonable choice of modification pattern.
\end{remark}

\section{Formulation of the solution via the PMP}\label{sec:PMP}

The optimization problems described in \eqref{eq:dynamic-receive}, \eqref{eq:dynamic-receive-simplified} and~\eqref{eq:dynamic-receive-simplified-undirected} are instances of a general optimal control problem with constraints imposed on the control variable. To outline our solution approach, we first recall the general formulation of an optimal control problem.

Consider a system characterized by a \textit{state} variable \( \mathbf{x}(t) \in \mathbb{R}^n \), which evolves according to the following system of differential equations:
    \begin{equation*}
        \begin{cases}
            \dot{\mathbf{x}}(t) = \mathbf{f}(\mathbf{x}(t), \mathbf{u}(t), t), \\
            \mathbf{x}(0) = \mathbf{x}_0,
        \end{cases}
    \end{equation*}
where \( \mathbf{u}(t) \in \mathcal{U}_{ad} \subset \mathbb{R}^m \) is the \textit{control} variable which is taken into a subset $\mathcal{U}_{ad}$ of $\mathbb{R}^m$ of admissible controls, and \( \mathbf{f} : \mathbb{R}^n \times \mathbb{R}^m \times \mathbb{R} \to \mathbb{R}^n \) is a continuous function defining the system dynamics. Our objective is to determine an optimal control function \( \mathbf{u}^*(t) \) that minimizes a given cost functional \( \mathrm{J} \), which typically takes the form:
    \begin{equation*}
        \mathrm{J}(\mathbf{x},\mathbf{u}) = \int_{0}^{T} L(\mathbf{x}(t), \mathbf{u}(t), t) \, \mathrm{d}t + \Phi(\mathbf{x}(T)),
    \end{equation*}
where \( L : \mathbb{R}^n \times \mathbb{R}^m \times \mathbb{R} \to \mathbb{R} \) represents the running cost, and \( \Phi : \mathbb{R}^n \to \mathbb{R} \) is the terminal cost. The optimal control function \( \mathbf{u}^*(t) \) must ensure that the state trajectory \( \mathbf{x}(t) \) satisfies the given dynamic constraints and initial conditions while minimizing the cost functional.

For a.e. \(t\in[0,T]\), we work with effective admissible control sets that are nonempty, closed, convex, and compact in all the cases under consideration. In addition, the corresponding state dynamics are measurable with respect to \(t\) and continuously differentiable with respect to the variables \((\mathbf r,U)\) on the admissible set. Hence, the PMP applies and
provides first-order necessary optimality conditions; see, e.g.,
\cite[Theorem 22.26]{clarke2013functional}. It provides necessary conditions that an optimal control \( \mathbf{u}^*(t) \) and corresponding state \( \mathbf{x}^*(t) \) must satisfy. Define the Hamiltonian function \( \mathrm{H} : \mathbb{R}^n \times \mathbb{R}^m \times \mathbb{R}^n \times \mathbb{R} \to \mathbb{R} \) as:
\begin{equation}\label{eq:generic_hamiltonian}
    \mathrm{H}(\mathbf{x}, \mathbf{u}, \boldsymbol{\lambda}, t) = \boldsymbol{\lambda}(t)^\top \mathbf{f}(\mathbf{x}, \mathbf{u}, t) + L(\mathbf{x}, \mathbf{u}, t),
\end{equation}
where \( \boldsymbol{\lambda}(t) \in \mathbb{R}^n \) is the \textit{costate} (or adjoint) variable.

The PMP asserts that for an optimal control \( \mathbf{u}^*(t) \) and state trajectory \( \mathbf{x}^*(t) \), there exists a nontrivial costate vector \( \boldsymbol{\lambda}(t) \) such that:
\begin{enumerate}
    \item The state and costate dynamics satisfy:
    \begin{eqnarray}
        & \dot{\mathbf{x}}^*(t)  = & \frac{\partial \mathrm{H}}{\partial \boldsymbol{\lambda}}(\mathbf{x}^*(t), \mathbf{u}^*(t), \boldsymbol{\lambda}(t), t), \label{eq:condition_on_state}\\
        & \dot{\boldsymbol{\lambda}}(t) = & -\frac{\partial \mathrm{H}}{\partial \mathbf{x}}(\mathbf{x}^*(t), \mathbf{u}^*(t), \boldsymbol{\lambda}(t), t). \label{eq:condition_on_costate}
    \end{eqnarray}
    \item The optimal control \( \mathbf{u}^*(t) \) minimizes the Hamiltonian:
    \begin{equation}\label{eq:hamiltonian_minimization}
        \mathbf{u}^*(t) = \arg \min_{\mathbf{u} \in \mathcal{U}_{ad}} \mathrm{H}(\mathbf{x}^*(t), \mathbf{u}, \boldsymbol{\lambda}(t), t),
    \end{equation}
    for the set of admissible controls $\mathcal{U}_{ad} \subset \mathbb{R}^m$, for which the unconstrained gradient is given by
    \begin{equation}\label{eq:hamiltonian_gradient}
        \mathbf{g}(t) = \frac{\partial \mathrm{H}}{\partial \mathbf{u}}.
    \end{equation}
    \item The boundary conditions include the initial condition \( \mathbf{x}(0) = \mathbf{x}_0 \), and if the terminal state \( \mathbf{x}(T) \) is free, the transversality condition for the costate:
    \begin{equation}\label{eq:terminal_condition}
        \boldsymbol{\lambda}(T) = \frac{\partial \Phi}{\partial \mathbf{x}}(\mathbf{x}(T)).
    \end{equation}
\end{enumerate}
In the following, we reformulate the problems \eqref{eq:dynamic-receive}, \eqref{eq:dynamic-receive-simplified} and~\eqref{eq:dynamic-receive-simplified-undirected} in a manner that allows us to leverage the PMP to derive a solution procedure. Specifically, we express these conditions as a constrained optimization problem:
\begin{equation*}
    \min_{\mathbf{u} \in \mathcal{U}_{ad}} \mathrm{J}(\mathbf{u}) = \min_{\mathbf{u} \in \mathcal{U}_{ad}} \mathrm{J}(\mathbf{x}(\mathbf{u}), \mathbf{u}) = \min_{\mathbf{x}, \mathbf{u}} \left\{ \mathrm{J}(\mathbf{x}, \mathbf{u}) \mid e(\mathbf{x}, \mathbf{u}) = 0 \right\},
\end{equation*}
where the constraint function $e(\mathbf{x}, \mathbf{u})$ is given by
\begin{equation*}
    e(\mathbf{x}, \mathbf{u}) = \begin{bmatrix}
        \dot{\mathbf{x}}(t) - \mathbf{f}(\mathbf{x}, \mathbf{u}, t) \\
        \mathbf{x}(0) - \mathbf{x}_0
    \end{bmatrix}.
\end{equation*}
The unconstrained gradient of the objective function is denoted by $\mathbf{g}(\mathbf{u})$. Suppose we have a projection operator $\Pi_{\mathcal{U}_{ad}}$ that projects a vector from $\mathbb{R}^m$ onto the feasible set $\mathcal{U}_{ad}$, i.e.,
\begin{equation*}
    \Pi_{\mathcal{U}_{ad}}(\mathbf{y}) = \arg\min_{\mathbf{x} \in \mathcal{U}_{ad}} \| \mathbf{x} - \mathbf{y}\|_2.
\end{equation*}
We can then construct a projected gradient method, 
incorporating Nesterov's acceleration 
to optimize this problem (Algorithm~\ref{alg:nesterov_pg}). 
\begin{algorithm}[htb]
\KwIn{Tolerance $\varepsilon_g > 0$, fixed step size $\eta > 0$, initial guess $\mathbf{u}^{(0)}$}
$\mathbf{u}^{(-1)} \gets \mathbf{u}^{(0)}$\;
$k \gets 0$\;
\While{not converged}{
    $\beta_k \gets \dfrac{k-1}{k+2}$\;
    $\mathbf{y}^{(k)} \gets \mathbf{u}^{(k)} + \beta_k\!\left(\mathbf{u}^{(k)} - \mathbf{u}^{(k-1)}\right)$\;
    \ShowLnLabel{line:solve_state} Solve~\eqref{eq:condition_on_state} at $\mathbf{y}^{(k)}$ to obtain $\mathbf{x}^{(k)}$\;
    \ShowLnLabel{line:solve_costate} Solve~\eqref{eq:condition_on_costate} at $(\mathbf{x}^{(k)},\mathbf{y}^{(k)})$ with terminal condition~\eqref{eq:terminal_condition} to obtain $\boldsymbol{\lambda}^{(k)}$\;
    Compute unconstrained gradient $\hat{\mathbf{g}}^{(k)}$ employing~\eqref{eq:hamiltonian_gradient}\;
    $\mathbf{u}^{(k+1)} \gets \Pi_{\mathcal{U}_{ad}}\!\left[ \mathbf{y}^{(k)} - \eta\,\hat{\mathbf{g}}^{(k)} \right]$\;
    \If{$\mathrm{J}(\mathbf{u}^{(k+1)}) > \mathrm{J}(\mathbf{y}^{(k)})$ \upshape(monotone restart)}{
        $\mathbf{y}^{(k)} \gets \mathbf{u}^{(k)}$, recompute $\hat{\mathbf{g}}^{(k)}$ at $\mathbf{u}^{(k)}$, and reset $\mathbf{u}^{(k+1)} \gets \Pi_{\mathcal{U}_{ad}}\!\left[\mathbf{u}^{(k)} - \eta\,\hat{\mathbf{g}}^{(k)}\right]$\;
    }
    \ShowLnLabel{line:project}$\mathbf{g}^{(k)} \gets \dfrac{1}{\eta}\!\left(\mathbf{y}^{(k)} - \mathbf{u}^{(k+1)}\right)$\;
    \lIf{$\dfrac{\|\mathbf{g}^{(k)}\|_2}{\max(1,\|\mathbf{y}^{(k)}\|_2)} \leq \varepsilon_g$}{\textbf{break}}
    $k \gets k + 1$\;
}
\caption{Nesterov accelerated projected gradient with monotone restart for the solution of an optimal control problem via the PMP.\label{alg:nesterov_pg}}
\end{algorithm}
The key idea is to maintain, alongside the iterates $\mathbf{u}^{(k)}$, an \emph{extrapolated} auxiliary sequence $\mathbf{y}^{(k)}$ that incorporates a momentum term from the previous step:
\begin{equation*}%
    \mathbf{y}^{(k)} = \mathbf{u}^{(k)} + \beta_k \left( \mathbf{u}^{(k)} - \mathbf{u}^{(k-1)} \right), \qquad \beta_k = \frac{k-1}{k+2}.
\end{equation*}
The gradient of the Hamiltonian and the state/costate equations are then evaluated at $\mathbf{y}^{(k)}$ rather than at $\mathbf{u}^{(k)}$, and the next iterate is obtained by a single projected gradient step from $\mathbf{y}^{(k)}$:
\begin{equation}\label{eq:nesterov_update}
    \mathbf{u}^{(k+1)} = \Pi_{\mathcal{U}_{ad}}\!\left[ \mathbf{y}^{(k)} - \eta\,\hat{\mathbf{g}}^{(k)} \right],
\end{equation}
where $\hat{\mathbf{g}}^{(k)}$ is the unconstrained gradient~\eqref{eq:hamiltonian_gradient} evaluated at $(\mathbf{x}(\mathbf{y}^{(k)}), \mathbf{y}^{(k)}, \boldsymbol{\lambda}(\mathbf{y}^{(k)}))$.

Since the optimization problem associated with the controlled dynamic network is not convex in general, the extrapolated point $\mathbf{y}^{(k)}$ may occasionally produce an accelerated step that increases the objective functional~\cite{ODonoghueCandesCMAA2015}. For this reason, we adopt a monotone restart strategy: if the candidate iterate obtained from~\eqref{eq:nesterov_update} satisfies
\[
    J\bigl(\mathbf{u}^{(k+1)}\bigr) > J\bigl(\mathbf{y}^{(k)}\bigr),
\]
then the momentum contribution is discarded for that iteration by resetting $\mathbf{y}^{(k)} \gets \mathbf{u}^{(k)}$, recomputing the corresponding gradient at $\mathbf{u}^{(k)}$, and replacing the accelerated trial with the standard projected-gradient step
\[
    \mathbf{u}^{(k+1)} = \Pi_{\mathcal{U}_{ad}}\!\left[ \mathbf{u}^{(k)} - \eta\,\hat{\mathbf{g}}^{(k)} \right].
\]
This safeguard preserves the simplicity of the fixed-step iteration while preventing the momentum term from generating non-descent updates.

The convergence of Algorithm~\ref{alg:nesterov_pg} is monitored via the \emph{projected-gradient mapping} norm, which is the standard optimality measure for projected-gradient methods~\cite[\S 17.1]{MR3364552}. At iteration $k$, given step size $\eta> 0$, the projected-gradient mapping is defined as
\[
    \mathbf{g}^{(k)} = \dfrac{1}{\eta}\!\left(\mathbf{y}^{(k)} - \mathbf{u}^{(k+1)}\right),
\]
defined in line~\ref{line:project} of the algorithm. Note that $\mathbf{g}^{(k)} = \mathbf{0}$ is a necessary optimality condition for the projected problem, and $\| \mathbf{g}^{(k)} \|_2$ measures how far the current iterate is from satisfying this condition. The algorithm is declared converged when
\[
    \frac{\|\mathbf{g}^{(k)}\|_2}{\max(1,\|\mathbf{y}^{(k)}\|_2)} \leq \varepsilon_g,
\]
where $\varepsilon_g > 0$ is a user-specified tolerance. As a secondary safeguard against stagnation, a simultaneous check on both the relative change in the iterates,
\[
    \frac{\|\mathbf{u}^{(k+1)} - \mathbf{u}^{(k)}\|_2}{\max(1,\|\mathbf{u}^{(k)}\|_2)} \leq \varepsilon_x,
\]
and the relative change in the objective value is also employed. 

In the following subsection, we derive the expression for the state, costate, and gradient equations for the problems we formulated in Section~\ref{sec:the-steering-problem}.

\subsection{The dynamic receive case}%

We start by applying the general theory to \emph{dynamic receive} case from~\eqref{eq:dynamic-receive}. For this problem the Hamiltonian~\eqref{eq:generic_hamiltonian} can be written as
\begin{equation}\label{eq:our-hamiltonian}
\begin{split}
  \mathrm{H}(\mathbf{r}, U, \boldsymbol{\lambda}, t) =  \frac{\alpha}{2} \| U(t) \|^2_F + \boldsymbol{\lambda}(t)^\top   \left( -b \left( \mathbf{r}(t) - \mathbf{1} \right) %
  - \log \left( I - a (A(t) + U(t))^\top \right) \mathbf{r}(t) \right).
\end{split}
\end{equation}
The next step is to derive the optimality conditions for the state~\eqref{eq:condition_on_state}, costate~\eqref{eq:condition_on_costate} and for the gradient of the minimization problem in~\eqref{eq:hamiltonian_minimization}. To obtain them we need to prove some preliminary results regarding the Fréchet derivative of the matrix logarithm function.

\begin{lemma}\label{lem:frechetting}
The partial derivative of the Hamiltonian $\mathrm{H}$ in~\eqref{eq:our-hamiltonian} with respect to the constraint $U$ can be expressed~as
\[
\left[ \frac{\partial \mathrm{H}}{\partial U} \right]_{i,j} =  \alpha \mathbf{e}_i^\top U \mathbf{e}_j + a \boldsymbol{\lambda}(t)^\top  \mathrm{L}_{\log}\left( I - a(A+U)^\top,E_{j,i} \right) \mathbf{r}(t) \qquad \forall\,i,j=1,\ldots,n.
\]
\end{lemma}

\begin{proof}
We call 
\[
f(U) = \frac{\alpha}{2}\| U \|_F^2 - \boldsymbol{\lambda}(t)^\top  \log\left( I - a(A + U)^\top\right) \mathbf{r}(t) = f_1(U) + f_2(U),
\]
with $f_1$ collecting the Frobenius norm term, and $f_2$ collecting the bilinear form involving the matrix-logarithm. Then $[\nicefrac{\partial \mathrm{H}}{\partial U}]_{i,j}$ can be expressed in terms of the rank-1 matrix $E_{i,j} = \mathbf{e}_i \mathbf{e}_j^\top$ as the sum of the derivatives of the two functions $f_1$ and $f_2$. Indeed,
\begin{equation*}
\begin{split}
    \lim_{t \rightarrow 0} & \frac{1}{t} \left( f_1(U + t E_{i,j}) - f_1(U) \right) = 
    \lim_{t \to 0}  \frac{\alpha}{2 t}\left( \operatorname{tr}((U+t E_{i,j})^\top (U+t E_{i,j})) - \operatorname{tr}(U^\top U) \right) = \\
    & \lim_{t \to 0}  \frac{\alpha}{2t} \left( \operatorname{tr}\left( t (E_{i,j}^\top U + U^\top E_{i,j} + t E_{i,j}^\top E_{j,i} \right) \right) = 
     \frac{\alpha}{2} \left( \operatorname{tr}\left( \mathbf{e}_j \mathbf{e}_i^\top U + U^\top \mathbf{e}_i \mathbf{e}_j^\top \right) \right),
\end{split}
\end{equation*}
exploiting again the linearity and the cyclicity of the trace we have
\[
\left[\frac{\partial f_1}{\partial U}\right]_{i,j} = \alpha \mathbf{e}_i^\top U \mathbf{e}_j, \qquad \forall\,i,j=1,\ldots,n.
\]
Let $M(U) := I-a(A+U)^\top$. We apply the same direct computation also in the case of $f_2$, obtaining
\begin{equation*}
\begin{split}
    \lim_{t \to 0} &  \frac{1}{t} \left( f_2(U + t E_{i,j}) - f_2(U) \right) = 
    \lim_{t \to 0}  - \frac{1}{t} \boldsymbol{\lambda}(t)^\top  \left( \log\left( M(U + t E_{i,j})\right) - \log\left( M(U)\right)  \right) \mathbf{r}(t),
\end{split}
\end{equation*}
now using the definition of Fréchet derivative (Definition~\ref{def:frechet}) we find
\begin{equation*}
\begin{split}
    \lim_{t \to 0} & - \frac{1}{t} \boldsymbol{\lambda}(t)^\top  \left( \log\left( M(U)\right) + \mathrm{L}_{\log}(M(U),-t a E_{j,i} ) + o(|t|\|E_{i,j}\|) %
    - \log\left( M(U)\right)\mathbf{r}(t) \right) = \\
    \lim_{t \to 0} & - \frac{1}{t} \boldsymbol{\lambda}(t)^\top  \left( \mathrm{L}_{\log}(M(U),-t a E_{j,i}) + o(|t|\|E_{i,j}\|) \right) \mathbf{r}(t),    
\end{split}
\end{equation*}
we now use the fact that the Fréchet derivative is linear in the second argument, to conclude that
\[
\left[\frac{\partial f_2}{\partial U}\right]_{i,j} = a \boldsymbol{\lambda}(t)^\top  \mathrm{L}_{\log}\left( I - a(A+U)^\top,E_{j,i} \right) \mathbf{r}(t) \qquad \forall\,i,j=1,\ldots,n.
\]
From which we recover the thesis by adding the two expressions for $f_1$ and~$f_2$.
\end{proof}

We can now express the entire derivative in matrix form using an expression for computing a bilinear form of the Fréchet derivative of a matrix function in a rank-1 direction~\cite{FrechetTrick}.
\begin{theorem}[{\cite[Theorem~2.3]{FrechetTrick}}]\label{thm:frechettrick}
Let $A \in \mathbb{R}^{n \times n}$, $\mathbf{u},\mathbf{v} \in \mathbb{R}^n$, let $f$ be Fréchet differentiable at $A$, and denote $E_{i,j} = \mathbf{e}_i \mathbf{e}_j^\top$. Then
\[
\mathbf{u}^\top \mathrm{L}_f(A,E_{i,j}) \mathbf{v} = \left[ \mathrm{L}_f(A^\top, \mathbf{u}\mathbf{v}^\top )\right]_{i,j}.
\]
\end{theorem}

With these preliminary results, we can then formulate the set of optimality conditions for the problem~\eqref{eq:dynamic-receive}. 
\begin{proposition}\label{pro:sufficient_condition}
The optimality conditions for the Hamiltonian~\eqref{eq:our-hamiltonian} with the linearized constraint from Proposition~\ref{prop:relaxed_constraint} are given by
\begin{equation*}
\begin{cases}
\begin{cases}
    \displaystyle \dot{\mathbf{r}}(t) = \frac{\partial \mathrm{H}}{\partial  \boldsymbol{\lambda}}  =   -b \left( \mathbf{r}(t) - \mathbf{1} \right) - \log \left( I - a (A(t) + U(t))^\top \right) \mathbf{r}(t), \\
    \mathbf{r}(0) = \mathbf{1}, 
\end{cases} \\
\begin{cases}\displaystyle  \dot{\boldsymbol{\lambda}}(t) =-\frac{\partial \mathrm{H}}{\partial \mathbf{r}}  =  b \boldsymbol{\lambda}(t) + \log \left( I - a (A(t) + U(t)) \right) \boldsymbol{\lambda}(t), \\
\boldsymbol{\lambda}(T) = \mathbf{r}(T) - \mathbf{r}^*.
\end{cases} \\[1.4em]
\displaystyle  \frac{\partial \mathrm{H}  }{\partial U} 
 =\alpha  U(t) + a\left(\mathrm{L}_{\log}\left(I - a(A(t)+U(t)),\boldsymbol{\lambda}(t) \mathbf{r}(t)^\top\right)\right)^\top, \\
U(t) \in \mathcal{U}_{\text{ad}}(t),
\end{cases}
\end{equation*}
where the Fréchet derivative $\mathrm{L}_{\log}(\cdot,\cdot)$ can be expressed in terms of the matrix logarithm following the formula \eqref{eq:explicit-frechet}, $\boldsymbol{\mu}(t) = (I - a A(t))^{-1} \mathbf{1}$, and $\mathcal{U}_{\text{ad}}(t)$ is the one given in~\eqref{eq:admset_original}.
\end{proposition}

\begin{proof}
The state equation for the dynamic receive $\mathbf{r}(t)$ is derived by noting that the Hamiltonian $\mathrm{H}$ is linear with respect to $\boldsymbol{\lambda}$. For the \emph{costate} equation, we proceed via direct computation as follows:
\[
\begin{aligned}
   \frac{\partial \mathrm{H}}{\partial \mathbf{r}} 
   &= \frac{\partial}{\partial \mathbf{r}} \left[ \frac{\alpha}{2} \| U(t) \|_F^2 + \boldsymbol{\lambda}(t)^\top  \left( -b \left( \mathbf{r}(t) - \mathbf{1} \right) - \log \left( I - a (A(t) + U(t))^\top \right) \mathbf{r}(t) \right) \right] \\
   &= -b \boldsymbol{\lambda}(t)^\top  - \boldsymbol{\lambda}(t)^\top  \log \left( I - a (A(t) + U(t))^\top \right),
\end{aligned}
\]
from which the result follows by transposing and changing the sign of the previous expression. 

To obtain the expression for $\frac{\partial \mathrm{H}}{\partial U}$, we utilize Lemma~\ref{lem:frechetting} and apply Theorem~\ref{thm:frechettrick} directly. %
Proposition~\ref{prop:relaxed_constraint} directly implies the well-posedness of the state and co-state equations for $\mathbf{r}(t)$ and $\boldsymbol{\lambda}(t)$, respectively. It therefore remains only to verify the Fréchet derivative term appearing in the gradient. 
The term $\mathrm{L}_{\log}\left(I - a(A(t)+U(t)),\boldsymbol{\lambda}(t) \mathbf{r}(t)^\top\right)$ can be equivalently expressed by~\eqref{eq:explicit-frechet} as
\[
\begin{bmatrix}
    I & O
\end{bmatrix}
\log\!\left(
\begin{bmatrix}
    I - a\bigl(A(t) + U(t)\bigr) & \boldsymbol{\lambda}(t)\,\mathbf{r}(t)^\top \\
    O & I - a\bigl(A(t) + U(t)\bigr)
\end{bmatrix}
\right)
\begin{bmatrix}
    O \\[4pt] I
\end{bmatrix}.
\]
The matrix logarithm of this $2 \times 2$ block matrix is well-defined under the condition
$
\rho\left(A(t) + U(t)\right) < \frac{1}{a},
$
which is ensured by Proposition~\ref{prop:relaxed_constraint}.
\end{proof}

In the linear case~\eqref{eq:dynamic-receive-simplified}, obtaining the optimality conditions is then an immediate corollary of Proposition~\ref{pro:sufficient_condition}.

\begin{corollary}%
The Hamiltonian for the optimal control~\eqref{eq:dynamic-receive-simplified} problem is defined~as
\[
\mathrm{H}(\mathbf{r}, U, \boldsymbol{\lambda}, t) = \frac{\alpha}{2}\| U(t) \|_F^2 + \boldsymbol{\lambda}(t)^\top \left( -b (\mathbf{r}(t) - \mathbf{1}) + a (A(t) + U(t))^\top \mathbf{r}(t) \right),
\]
and the optimality conditions are given by
\[
\begin{cases}
     \begin{cases}\displaystyle \dot{\mathbf{r}}(t) = \displaystyle \frac{\partial \mathrm{H}}{\partial \boldsymbol{\lambda}} = - b (\mathbf{r}(t) - \mathbf{1}) + a (A(t) + U(t))^\top \mathbf{r}(t),\\
     \displaystyle \mathbf{r}(0) = \displaystyle \mathbf{1}, \end{cases} \\
\begin{cases}\displaystyle  \dot{\boldsymbol{\lambda}}(t) =-\frac{\partial \mathrm{H}}{\partial \mathbf{r}}  =  b \boldsymbol{\lambda}(t) - a \left(A(t) + U(t)\right) \boldsymbol{\lambda}(t), \\
\boldsymbol{\lambda}(T) = \mathbf{r}(T) - \mathbf{r}^*.
\end{cases} \\[1.4em]
\displaystyle  \frac{\partial \mathrm{H}  }{\partial U} 
 = \alpha U(t) + a \mathbf{r}(t) \boldsymbol{\lambda}(t)^\top, \\
U(t) \in \mathcal{U}_{\text{ad}}(t),
\end{cases}
\]
for $\mathcal{U}_{\text{ad}}(t)$ in either~\eqref{eq:adm_set_directed} or~\eqref{eq:adm_set}.
\end{corollary}

At this stage, we have derived all the analytical expressions required to implement Algorithm~\ref{alg:nesterov_pg}. It remains to address their discretization and to develop an efficient computational strategy, which we discuss in the next Section~\ref{sec:discretizing_and_solving}.

\section{Discretizing and solving the problem}\label{sec:discretizing_and_solving}

To solve the problems in~\eqref{eq:dynamic-receive}, \eqref{eq:dynamic-receive-simplified} and~\eqref{eq:dynamic-receive-simplified-undirected} using the PMP framework introduced in Section~\ref{sec:PMP} and summarized in Algorithm~\ref{alg:nesterov_pg}, we first need to discretize the expressions for the gradient, as well as the state and costate equations. 

Once this is achieved, we apply a projected gradient method enhanced with the Nesterov's acceleration to compute an approximation of the optimal control $U(t)$ at the selected discretization time steps. In our implementation, all norms are induced by the Frobenius inner product on control matrices $U(t)$, i.e., aggregated across all time steps. We use the tolerances $\varepsilon_g $ and $\varepsilon_x$ for monitoring the convergence of the projected gradient Algorithm~\ref{alg:nesterov_pg}.

In the following sections, we detail the components of the resulting algorithm. Specifically, we begin by deriving an appropriate discretization scheme for the state and costate equations. We then address the efficient computation of matrix-function--vector products (Section~\ref{sec:krylov_stuff}), and finally discuss the evaluation of the projection operators (Section~\ref{sec:computing_projections}).

\subsection{Selecting a time-stepping method for the state and costate equations}\label{sec:discretize_state_and_costate}

To evaluate the conditions in Proposition~\ref{pro:sufficient_condition}, i.e., update the state and costate in lines~\ref{line:solve_state} and~\ref{line:solve_costate} of Algorithm~\ref{alg:nesterov_pg}, we need to march in time two non-autonomous~ODEs
\begin{equation}\label{eq:state_for_r}
    \begin{cases}
    \displaystyle \dot{\mathbf{r}}(t) = \frac{\partial \mathrm{H}}{\partial  \boldsymbol{\lambda}}  =   -b \left( \mathbf{r}(t) - \mathbf{1} \right) - \log \left( I - a (A(t) + U(t))^\top \right) \mathbf{r}(t), \\
    \mathbf{r}(0) = \mathbf{1}, 
\end{cases} 
\end{equation}
and
\begin{equation}\label{eq:costate_for_r}
    \begin{cases}\displaystyle  \dot{\boldsymbol{\lambda}}(t) =-\frac{\partial \mathrm{H}}{\partial \mathbf{r}}  =  b \boldsymbol{\lambda}(t) + \log \left( I - a (A(t) + U(t)) \right) \boldsymbol{\lambda}(t), \\
\boldsymbol{\lambda}(T) = \mathbf{r}(T) - \mathbf{r}^*.
\end{cases}
\end{equation}
In general, we consider the adjacency matrix $A(t)$ at a finite set of time instants, $\{A(t_k)\}_k$. However, these discrete observations are often insufficient to define a difference equation with reasonable accuracy, especially when the time intervals $h_k = t_{k+1} - t_k$ are large. To address this, it is typically assumed that $A(t)$ is a piecewise constant function over the intervals $[t_k, t_{k+1})$. While this assumption simplifies the representation, it results in the dynamics of the corresponding ODE being only piecewise continuous. 

To effectively integrate the ODE systems \eqref{eq:state_for_r}--\eqref{eq:costate_for_r}, it is necessary to account for the discontinuities in the time-dependent adjacency matrix $A(t)$. Under the piecewise constant assumption, the right-hand side of both equations is only piecewise continuous, with possible jumps at the time instants $\{t_k\}_k$.

To handle this structure, we adopt an explicit Euler scheme applied separately on each interval $[t_k, t_{k+1})$ with a fixed step size $h$. More precisely, within each interval, the dynamics are integrated assuming $A(t)$ is constant and equal to $A(t_k)$. The numerical solution is then advanced step by step until reaching $t_{k+1}$.

At each switching time $t_{k+1}$, the value of the numerical solution obtained at the end of the interval $[t_k, t_{k+1})$ is used as the initial condition for the subsequent interval $[t_{k+1}, t_{k+2})$. In this way, the integration procedure consistently propagates the solution across intervals while properly accounting for the discontinuities in $A(t)$. While we selected the Euler method to control the computational cost, higher-order time-stepping methods can in principle improve the accuracy of the approximation. %

\subsection{Evaluating the matrix logarithm and its Fréchet derivative}\label{sec:krylov_stuff}

To apply the procedure to a large time-dependent network we need a reliable way to evaluate for a large, sparse matrices \(M\) the matrix-vector product \(\log(M)\mathbf{v}\). Instead of explicitly forming \(\log(M)\), we employ Krylov-type subspace methods to provide an efficient way to approximate \(\log(M)\mathbf{v}\).

The key idea behind Krylov methods is to approximate the action of \(\log(M)\) on \(\mathbf{v}\) by working on the Krylov subspace of dimension \(k\) 
\[\mathcal{K}_k(M,\mathbf{v}) = \{\mathbf{v}, M\mathbf{v}, M^2\mathbf{v}, \ldots, M^{k-1}\mathbf{v}\}.\]
The approximation in this subspace becomes more manageable by working on a smaller, projected, version of \(M\). The process begins by constructing an orthonormal basis $Q_k$ for the Krylov subspace via the Lanczos method~\cite[Algorithm~6.15]{SaadBook} for a symmetric $M$ or via the Arnoldi method~\cite[Algorithm~6.2]{SaadBook} for a general matrix. The result is a set of basis vectors that are used to form a small tridiagonal or Hessenberg matrix \(H_k\), respectively, which represents \(M\) in the reduced subspace as $H_k = Q_k^\top M Q_k$.  Once the projection is complete, we compute the matrix logarithm of \(H_k\) explicitly~\cite{MR3080997,MR1824061} and obtain the approximation 
\[
\log(M)\mathbf{v} \approx Q_k \log(H_k) (Q_k^\top \mathbf{v}).
\]
For the large number of evaluations that we need to march in time the solution to the state, costate equation and Hamiltonian gradient in Proposition~\ref{pro:sufficient_condition} polynomial Krylov methods are particularly appealing. To march in time the solution of the differential equation for the state and the costate we can resort to the Lanczos method in the case where the underlying network is undirected at every instant of time and we seek a symmetric control $U(t)$ to preserve the overall structure.

To evaluate the gradient of the Hamiltonian we need to compute the Fr\'echet derivative of the matrix logarithm. By~\eqref{eq:explicit-frechet}, the Fr\'echet derivative $\mathrm{L}_{\log}(M,E)$ can be extracted from the matrix logarithm of a $2n \times 2n$ block matrix. For each direction vector $\mathbf{v}$, this amounts to computing the $n$-dimensional vector
\[
\mathrm{L}_{\log}(M,E)\mathbf{v} = \begin{bmatrix} I & O \end{bmatrix} \log \begin{pmatrix} M & E \\ O & M \end{pmatrix} \begin{bmatrix} \mathbf{0} \\ \mathbf{v} \end{bmatrix}, \quad M = I - a(A(t_k) + U_k),
\]
for $U_k$ the approximation of $U(t_k)$ induced by the choice of the time-stepping procedure. Although polynomial Krylov methods perform well for the state and costate equations, where the matrix $M$ has spectrum in the right half-plane and the Arnoldi projection preserves this property reliably, the situation is different for the $2n \times 2n$ block matrix above. For this doubled system, the projected Hessenberg matrix $H_k$ can develop eigenvalues with negative real part, even when the spectrum of the original block matrix lies entirely in the right half-plane. Since the principal logarithm has a branch cut on $(-\infty,0]$, such spurious eigenvalues render the evaluation of $\log(H_k)$ unreliable. To address this, we employ \emph{rational} Krylov methods~\cite{Ruhe1998,GuettelKnizhnerman2013}, which replace the polynomial Krylov subspace by the rational Krylov subspace
\[
\mathcal{Q}_k(M,\mathbf{b},\Xi) = \operatorname{span}\left\{ \prod_{i=1}^{k}(M - \xi_i I)^{-1} \, p(M)\mathbf{b} \;:\; p \in \mathcal{P}_{k-1} \right\},
\]
where $\Xi = \{\xi_1,\ldots,\xi_k\}$ is a prescribed set of \emph{poles}. Poles at infinity recover the standard polynomial Krylov step. The rational Arnoldi decomposition~\cite{BerljafaGuettel2015} produces an orthonormal basis $V_k$ and upper-Hessenberg pencil $(K_k,H_k)$ such that
\[
M \, V_k K_k = V_k H_k.
\]
The small projected matrix $\hat M_k = K_k^{-1} H_k$ captures the spectral information of $M$ relevant to the starting vector, and the approximation reads
\begin{equation}\label{eq:rat_krylov}
    \log(M)\mathbf{b} \approx V_k\, K_k \log(\hat M_k)\, (K_k^{-1} e_1 \|\mathbf{b}\|_2).
\end{equation}
In practice, the quality of the rational approximation depends crucially on the choice of poles. We select them automatically by computing an AAA rational approximant~\cite{NakatsukasaSeteEtAl2018} of $\log(x)$ on a set of sample points covering the interval $[1 - a \max_t \rho(A(t)), 1 + a \max_t \rho(A(t))]$, where $a > 0$ is the edge attenuation parameter of the network model. The poles of the resulting rational function provide near-optimal shifts for the rational Krylov iteration. In this way, the number of poles (and therefore the Krylov subspace dimension) is determined adaptively by the AAA algorithm and is typically small (on the order of $10$--$20$) even when polynomial Krylov methods would require significantly more iterations.

We apply this rational Krylov strategy to the $2n \times 2n$ block matrix from~\eqref{eq:explicit-frechet}. The block upper-triangular structure of this matrix is exploited in the operator-vector products and shifted system solves required by the rational Arnoldi iteration: each shifted solve reduces to two triangular back-substitutions with the $n\times n$ shifted matrix $\nu M - \mu I$, preserving the sparsity of the original problem.

\subsection{Computing the projections}\label{sec:computing_projections}

The projected gradient method in Algorithm~\ref{alg:nesterov_pg} requires the computation of the projection operator
$\Pi_{\mathcal{U}_{ad}}:\mathbb{R}^{m}\rightarrow\mathcal{U}_{ad}$ onto the admissible control set. The dimension $m$ corresponds to the number of independent control variables after taking into account the sparsity pattern $\mathcal{P}$. In the general directed case, $m$ is equal to the number of admissible entries of the pattern. In the symmetric case~\eqref{eq:adm_set}, the number of variables is further reduced by considering only one triangular part of the control matrix, since the remaining entries are imposed by the condition $U(t)=U^\top(t)$.

The sparsity constraint
$[U(t)]_{i,j}=0$ if $(i,j)\notin\mathcal{P}$
is therefore incorporated directly into the choice of optimization variables and does not need to be enforced explicitly during the projection step.

For the admissible sets~\eqref{eq:adm_set_directed} and~\eqref{eq:adm_set}, the constraint
$U(t)\geq -A(t)$
is an elementwise lower bound. After discretizing the control variables in time and introducing the vectorization operator $\operatorname{vec}(\cdot)$, the projection onto this constraint reduces to a componentwise clipping operation.

Let $\mathbf{U}^{(k)}$ denote the vectorized control at iteration $k$ and let $G^{(k)}$ denote the gradient of the objective functional. The projected gradient update is
\begin{align*}
\mathbf{U}^{(k+1)}
&=
\Pi_{\mathcal{U}_{ad}}
\left(
\mathbf{U}^{(k)}
-\eta G^{(k)}
\right)
\nonumber \\
&=
\max
\left\{
\mathbf{U}^{(k)}
-\eta
\left(
\alpha\mathbf{U}^{(k)}
+
a\,\operatorname{vec}
(\boldsymbol{\lambda}\mathbf{r}^{\top})
\right),
-\operatorname{vec}(\mathbf{A})
\right\},
\end{align*}
where the maximum is taken componentwise. Here $\mathbf{A}$ denotes the vectorized adjacency matrix corresponding to the considered time step.

For the admissible set~\eqref{eq:admset_original}, the constraint involving the Katz vector is given by
\[
U(t)\boldsymbol{\mu}(t)
\leq
\left(\frac{1}{a}-\epsilon\right)\mathbf{1}.
\]
Using the identity $\operatorname{vec}(U\boldsymbol{\mu}) = (\boldsymbol{\mu}^{\top}\otimes I_n)\operatorname{vec}(U)$,
the constraint can be written as a linear inequality in the vectorized control variables:
\begin{align*}
(\boldsymbol{\mu}(t)^\top\otimes I_n)\mathbf{U}(t)
\leq
\left(\frac{1}{a}-\epsilon\right)\mathbf{1}.
\end{align*}
The unconstrained gradient step is first computed independently at each time instant. E.g., for the logarithmic dynamics, this gives
\begin{align*}
\widetilde{\mathbf{U}}^{(k)}
=
\mathbf{U}^{(k)}
-
\eta
\left[
\alpha\mathbf{U}^{(k)}
+
a\,\operatorname{vec}
\left(
\mathrm{L}_{\log}
\left(
I-a(A(t_k)+U^{(k)})^\top,
\boldsymbol{\lambda}(t_k)\mathbf{r}(t_k)^\top
\right)
\right)
\right].
\end{align*}
The projection is then obtained by solving the quadratic program
\begin{align*}
\min_{\mathbf{U}\in\mathbb{R}^{m}}
\quad&
\frac{1}{2}
\left\|
\mathbf{U}-\widetilde{\mathbf{U}}^{(k)}
\right\|_2^2
\\
\text{s.t.}\quad&
(\boldsymbol{\mu}(t_k)^\top\otimes I_n)\mathbf{U}
\leq
\left(\frac{1}{a}-\epsilon\right)\mathbf{1},
\\
&
\mathbf{U}\geq-\mathbf{A}(t_k).
\end{align*}
This projection is a convex quadratic program because both the objective and the constraints are convex. The case in~\eqref{eq:adm_set_directed}
can be handled in an entirely analogous manner, leading to a similar quadratic program where the first linear constraint is replaced by $(\mathbf{1}^{\top} \otimes I_n)\mathbf{U} \leq \left(\frac{\pi}{a}-\epsilon\right)\mathbf{1} - A(t_k)\mathbf{1}$.
Finally, for the symmetric admissible set~\eqref{eq:adm_set}, the equality constraint $U(t)=U^\top(t)$ is enforced by a reduced parametrization. Let $\mathcal{I}_{\mathcal{P}}$ denote the set of admissible entries in the upper triangular part of the sparsity pattern. The independent control variables are collected in
$
\mathbf{u}(t)
=
\{U_{ij}(t):(i,j)\in\mathcal{I}_{\mathcal{P}},\ i\leq j\}$. 
A linear expansion operator $S$ reconstructs the complete symmetric matrix according to $
\operatorname{vec}(U(t))
=
S\mathbf{u}(t)$. The projected gradient step is therefore performed in the reduced variable space:
$\mathbf{u}^{(k+1)}
=
\Pi_{\mathcal{U}_{ad}^{s}}
\left(
\mathbf{u}^{(k)}
-\eta g^{(k)}
\right)$, where the reduced gradient is obtained from the full matrix gradient through
$g^{(k)}
=
S^\top
\operatorname{vec}
\left(
\nabla_U J(U^{(k)})
\right)$.
For the elementwise admissible constraint, the projection remains a componentwise clipping operation, but it is applied only to the independent triangular variables
\begin{align*}
u_{ij}^{(k+1)}
=
\max
\left\{
u_{ij}^{(k)}
-\eta g_{ij}^{(k)},
-A_{ij}(t_k)
\right\},
\qquad
(i,j)\in\mathcal{I}_{\mathcal{P}}.
\end{align*}
The remaining entries are recovered automatically by symmetry, namely
$U_{ji}=U_{ij}$.
Consequently, the equality constraint $U=U^\top$ is satisfied by construction and no additional projection step is required.

\section{Numerical examples}\label{sec:numerical_examples}

In this section we consider numerical experiments to validate the theory discussed in Section~\ref{sec:PMP}. In Section~\ref{sec:small_examples} we discuss a few small-size cases for which we can easily visualize the obtained controls and states to verify that the results are the expected ones. Then, in Section~\ref{sec:large_examples}, we investigate a larger case scenario in which we employ the techniques discussed in Section~\ref{sec:discretizing_and_solving} to reduce the computational costs. %
The implementation is done in Matlab (\texttt{v9.10.0.1602886 (R2021a)}), and the experiments are run on a single node (AMD EPYC 7763, 2 Threads per core, 64 cores per socket, 2 sockets, 2 \si{\tera\byte} of RAM, Ubuntu 22.04.4 LTS) of the Toeplitz cluster at the Green Data Center (GDC) of the University of Pisa. For the AAA algorithm~\cite{NakatsukasaSeteEtAl2018} we employ the implementation available in the Chebfun~\cite{driscoll2014chebfun} package (\texttt{v5.7.0}). For the generation of the Rational Krylov subspace we use instead the Rational Krylov Toolbox for Matlab (RKToolbox)~\cite{BerljafaGuettel2015} (\texttt{v2.9}).
\subsection{Small descriptive examples}\label{sec:small_examples}

The focus of this section is on small-scale examples where the results can be effectively visualized. We begin with the case of the temporal network illustrated in Figure~\ref{fig:switching-tree-shortcut}.
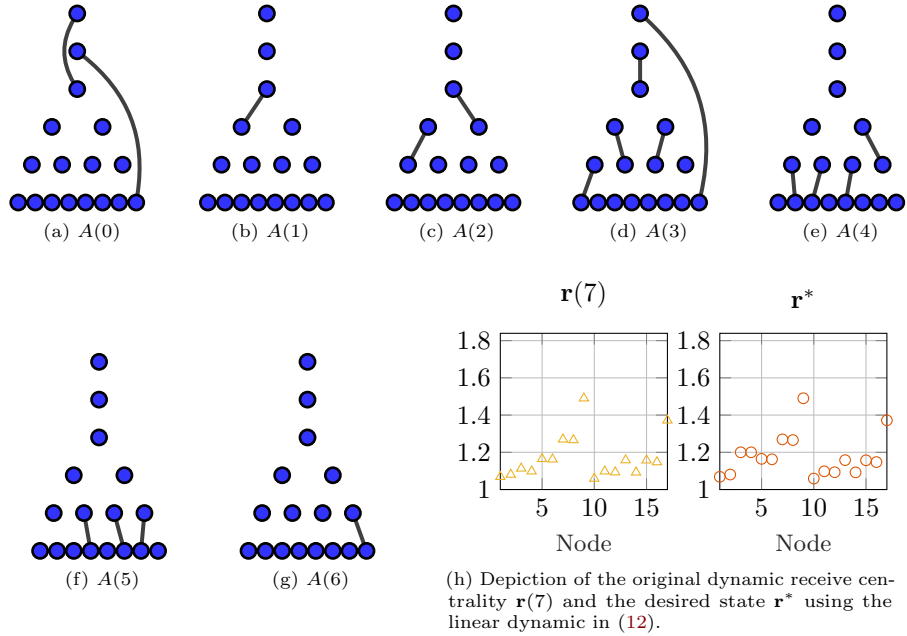
\begin{figure}[htb]
    \centering
    \subfloat[$A(0)$]{\input{phonecall2_1}}\hfil
    \subfloat[$A(1)$]{\input{phonecall2_2}}\hfil
    \subfloat[$A(2)$]{\input{phonecall2_3}}\hfil
    \subfloat[$A(3)$]{\input{phonecall2_4}}\hfil
    \subfloat[$A(4)$]{\input{phonecall2_5}}
    
    \subfloat[$A(5)$]{\input{phonecall2_6}}\hfil
    \subfloat[$A(6)$]{\input{phonecall2_7}}\hfil
    \subfloat[Depiction of the original dynamic receive centrality $\mathbf{r}(7)$ and the desired state $\mathbf{r}^*$ using the linear dynamic in~\eqref{eq:dynamic-receive-simplified}.\label{fig:switching-tree-shortcut:states}]{\input{phonecalltwodesiredstate1}}

    \caption{Example from~\cite[Figure~3]{HighamNet}, links are active over non-overlapping time intervals cycling $t \mod{7}$. We consider a modified target state $\mathbf{r}^*$ which is $\mathbf{r}^* \equiv \mathbf{r}(7)$ with $\mathbf{r}_3^* = \mathbf{r}_4^* = 1.2$.}
    \label{fig:switching-tree-shortcut}
\end{figure}
This network, originally presented in~\cite{HighamNet}, is a synthetic telephone-call example with $n = 17$ nodes evolving through seven snapshots. The nodes represent callers and are visibly organized in the form of a tree, at time step $0$ two callers start the diffusion of an information that triggers a cascade of calls between the callers. We first compute the uncontrolled dynamic receive centrality $\mathbf{r}(t)$ over $t \in [0,7]$, which measures the importance as receiver at time $t$ of the different callers, then we define a perturbed target by setting $\mathbf{r}^* \equiv \mathbf{r}(7)$, $\mathbf{r}_3^* = \mathbf{r}_4^* = 1.2$,
while keeping all other components equal to the uncontrolled final state (Figure~\ref{fig:switching-tree-shortcut:states}). Namely, we want to just increase the importance as receivers of two of the callers.

The admissible pattern is
$
\mathcal{P} = \{ (i,j) \in V\times V\, :\, (i,j) \in E_t \text{ for some } t\in[0,T]\}$,
that is, every edge that appears at least once in the temporal sequence can be controlled. Time integration is performed with the forward Euler scheme discussed in Section~\ref{sec:discretize_state_and_costate}, using step size $h=10^{-2}$. We then solve problem~\eqref{eq:dynamic-receive} by Algorithm~\ref{alg:nesterov_pg} for three values of the Tikhonov weight $\alpha\in\{1,0.5,0.05\}$, and $\eta = 10^{-1}$.
 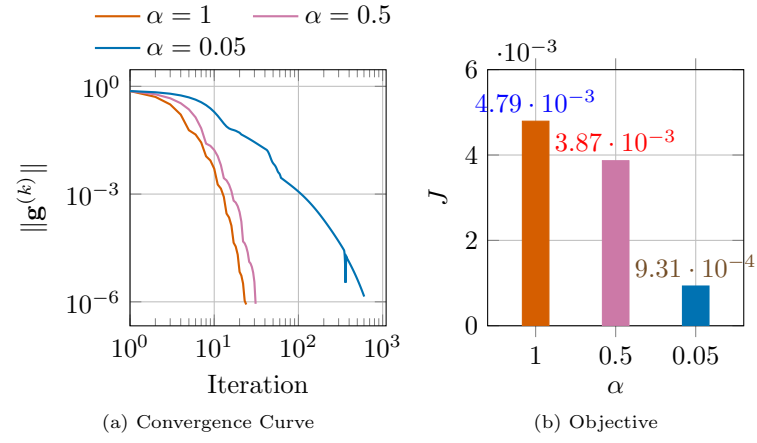
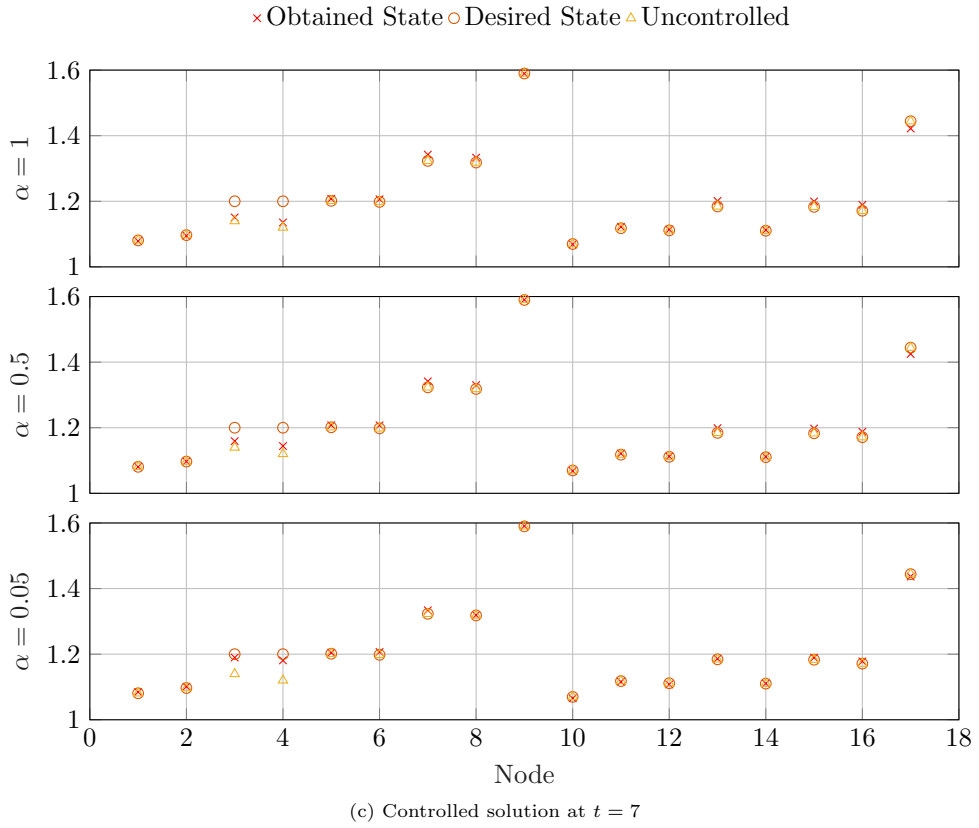
\begin{figure}[htbp]
     \centering

    \subfloat[Convergence Curve\label{fig:convergence_curve_small_log}]{
\begin{tikzpicture}
\begin{axis}[
    width=0.38\columnwidth,
    height=0.38\columnwidth,
    xlabel={Iteration},
    ylabel={$\|\mathbf{g}^{(k)}\|$},
    legend pos=north east,
    grid=major,
    ymode=log,
    xmode=log,
    xmin=1,
    xmax=1100,
    legend columns = 2,
    legend style={
        at={(0.45,1)},
        legend cell align=left,
        align=left,
        anchor=south,
        draw=none,
        fill=none,
    }
]
\addplot[
    blind1,
    thick,
] table[
    x index=0,
    y index=1,
] {convergence_alpha1.dat};
\addlegendentry{$\alpha = 1$}
\addplot[
    blind2,
    thick,
] table[
    x index=0,
    y index=1,
] {convergence_alpha05.dat};
\addlegendentry{$\alpha = 0.5$}
\addplot[
    blind3,
    thick,
] table[
    x index=0,
    y index=1,
] {convergence_alpha005.dat};
\addlegendentry{$\alpha = 0.05$}
\end{axis}
\end{tikzpicture}
    }
 \subfloat[Objective\label{fig:objective_small_log}]{%
\begin{tikzpicture}
\begin{axis}[
    width=0.38\columnwidth,
    height=0.38\columnwidth,
    ybar,
    bar width=10pt,
    ylabel={$J$},
    xlabel={$\alpha$},
    xtick={1,2,3},
    xticklabels={1,0.5,0.05},
    ymin=0,
    ymax=6e-3,
    enlarge x limits=0.3,
    grid=major,
]

\addplot+[ybar, bar shift=0pt, fill=blind1, draw=blind1,
          nodes near coords] coordinates {(1,4.792506e-03)};

\addplot+[ybar, bar shift=0pt, fill=blind2, draw=blind2,
          nodes near coords] coordinates {(2,3.866826e-03)};

\addplot+[ybar, bar shift=0pt, fill=blind3, draw=blind3,
          nodes near coords] coordinates {(3,9.314164e-04)};

\end{axis}
\end{tikzpicture}
}

\subfloat[Controlled solution at $t = 7$\label{fig:dynamic-receive-panel-c}]{
\input{solution-phone}
}
     
    \caption{Example taken from \cite[Figure~3]{HighamNet}. Panel~\ref{fig:convergence_curve_small_log} displays the gradient norm history, while panel~\ref{fig:objective_small_log} shows the value of the cost functional $J$ at convergence. Panel~\ref{fig:dynamic-receive-panel-c} compares the target state, the uncontrolled dynamics, and the controlled dynamics for problem~\eqref{eq:dynamic-receive}.}
     \label{fig:switching-tree-shortcut:results}
 \end{figure}
The convergence curves in Figure~\ref{fig:switching-tree-shortcut:results}(a) show the expected trade-off between control effort and iteration count. For $\alpha=1$ the method reaches $\|\mathbf{g}^{(k)}\|\approx 8.43\times 10^{-7}$ in $24$ iterations, while for $\alpha=0.5$ it reaches $\|\mathbf{g}^{(k)}\|\approx 8.80\times 10^{-7}$ in $31$ iterations; the low-regularization case $\alpha=0.05$ requires substantially more iterations ($604$) and reaches $\|\mathbf{g}^{(k)}\|\approx 1.40\times 10^{-6}$.

The objective values in Figure~\ref{fig:switching-tree-shortcut:results}(b) are consistent with this behavior: reducing $\alpha$ yields a lower final cost, from $J\approx 4.80\times 10^{-3}$ ($\alpha=1$) to $J\approx 3.87\times 10^{-3}$ ($\alpha=0.5$) and $J\approx 9.31\times 10^{-4}$ ($\alpha=0.05$). Panel~\ref{fig:dynamic-receive-panel-c} confirms that the controlled terminal state moves the two modified components toward the prescribed target, and that lower regularization improves the terminal matching while preserving the qualitative structure of the uncontrolled profile on the remaining components.

We repeat the same test for the linear dynamics in~\eqref{eq:dynamic-receive-simplified}, using exactly the same target-design procedure as above. Namely, we first compute the uncontrolled trajectory up to $t=7$, then define the desired terminal state by replacing only the components at nodes $3$ and $4$ with the value $1.2$, while leaving all other entries unchanged. Figure~\ref{fig:switching-tree-shortcut-linear:results} reports the corresponding convergence histories and final objective values.
 \begin{figure}[htbp]
     \centering

    \subfloat[Convergence Curve\label{fig:switching-tree-shortcut-linear:results-a}]{
\begin{tikzpicture}
\begin{axis}[
    width=0.38\columnwidth,
    height=0.38\columnwidth,
    xlabel={Iteration},
    ylabel={$\|\mathbf{g}^{(k)}\|$},
    legend pos=north east,
    grid=major,
    ymode=log,
    xmode=log,
    xmin=1,
    xmax=1500,
    legend columns = 2,
    legend style={
        at={(0.45,1)},
        legend cell align=left,
        align=left,
        anchor=south,
        draw=none,
        fill=none,
    }
]
\addplot[
    blind1,
    thick,
] table[
    x index=0,
    y index=1,
] {convergence_linear_alpha1.dat};
\addlegendentry{$\alpha = 1$}
\addplot[
    blind2,
    thick,
] table[
    x index=0,
    y index=1,
] {convergence_linear_alpha05.dat};
\addlegendentry{$\alpha = 0.5$}
\addplot[
    blind3,
    thick,
] table[
    x index=0,
    y index=1,
] {convergence_linear_alpha005.dat};
\addlegendentry{$\alpha = 0.05$}
\end{axis}
\end{tikzpicture}
    }
 \subfloat[Objective\label{fig:switching-tree-shortcut-linear:results-b}]{%
\begin{tikzpicture}
\begin{axis}[
    width=0.38\columnwidth,
    height=0.38\columnwidth,
    ybar,
    bar width=10pt,
    ylabel={$J$},
    xlabel={$\alpha$},
    xtick={1,2,3},
    xticklabels={1,0.5,0.05},
    ymin=0,
    ymax=9e-3,
    enlarge x limits=0.3,
    grid=major,
]

\addplot+[ybar, bar shift=0pt, fill=blind1, draw=blind1,
          nodes near coords] coordinates {(1,7.350360e-03)};

\addplot+[ybar, bar shift=0pt, fill=blind2, draw=blind2,
          nodes near coords] coordinates {(2,6.117794e-03)};

\addplot+[ybar, bar shift=0pt, fill=blind3, draw=blind3,
          nodes near coords] coordinates {(3,1.924513e-03)};

\end{axis}
\end{tikzpicture}%
}

     \caption{Example taken from \cite[Figure~3]{HighamNet}. Panel~\ref{fig:switching-tree-shortcut-linear:results-a} displays the gradient norm history, while panel~\ref{fig:switching-tree-shortcut-linear:results-b} shows the value of the cost functional $J$ at convergence. \label{fig:switching-tree-shortcut-linear:results}}
 \end{figure}
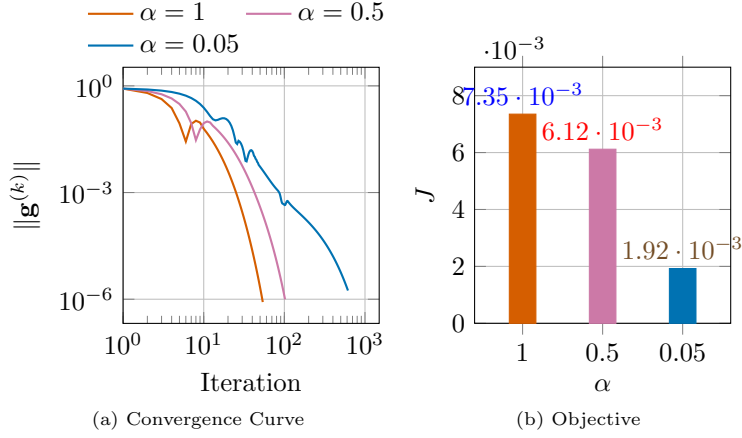
The linear case exhibits the same qualitative dependence on the Tikhonov weight $\alpha$ observed in the logarithmic model. From Figure~\ref{fig:switching-tree-shortcut-linear:results-a}, for $\alpha=1$ the method converges in $54$ iterations with $\|\mathbf{g}^{(k)}\|\approx 8.33\times 10^{-7}$, for $\alpha=0.5$ in $102$ iterations with $\|\mathbf{g}^{(k)}\|\approx 9.93\times 10^{-7}$, and for $\alpha=0.05$ in $615$ iterations with $\|\mathbf{g}^{(k)}\|\approx 1.75\times 10^{-6}$. As expected, weaker regularization reduces the final value of the cost functional: Figure~\ref{fig:switching-tree-shortcut-linear:results-b} gives $J\approx 7.35\times 10^{-3}$ ($\alpha=1$), $J\approx 6.12\times 10^{-3}$ ($\alpha=0.5$), and $J\approx 1.92\times 10^{-3}$ ($\alpha=0.05$). Overall, this confirms the same accuracy-versus-effort trade-off: smaller $\alpha$ yields better terminal matching at the price of more iterations. %

\subsection{Larger scale example}\label{sec:large_examples}

As a larger-scale example, we consider the \texttt{CollegeMsg} temporal network~\cite{Panzarasa2009911}, a dataset of private messages exchanged on an online social network at the University of California, Irvine. The dataset records directed interactions: an edge $(u, v, t)$ indicates that user $u$ sent a private message to user $v$ at Unix timestamp\footnote{The Unix timestamp is the number of seconds between a particular date and the Unix Epoch (January 1\textsuperscript{st}, 1970 UTC).} $t$, yielding a total of $n = 1899$ users and over $59\,000$ timestamped communications.

To construct the time-varying adjacency matrices, we partition the full observation period uniformly into $n_t = 30$ time windows and normalize the resulting breakpoints to the interval $[0, 10]$; see Figure~\ref{fig:messagenet-window}. For each window $i = 1, \ldots, n_t - 1$, we collect all edges with timestamps in $[t_i, t_{i+1})$ and assemble the sparse adjacency matrix $A_i \in \mathbb{R}^{n \times n}$, where each nonzero entry records an observed message. The communication activity varies considerably across windows, as reflected by the spread in the number of nonzeros $\operatorname{nnz}(A_i)$ shown in Figure~\ref{fig:messagenet-nnz}. The parameter $a$ in the centrality model is set to $a = 1/(\rho_{\max} + 1)$, where $\rho_{\max} = \max_i \rho(A_i)$ is the largest spectral radius across all snapshots, ensuring that $I - aA_i$ admits a matrix logarithm throughout the horizon; see again Remark~\ref{rem:matrix_log_existence}.
\begin{figure}[htbp]
    \centering
    \subfloat[Uniform time partition over the observation period\label{fig:messagenet-window}]{\input{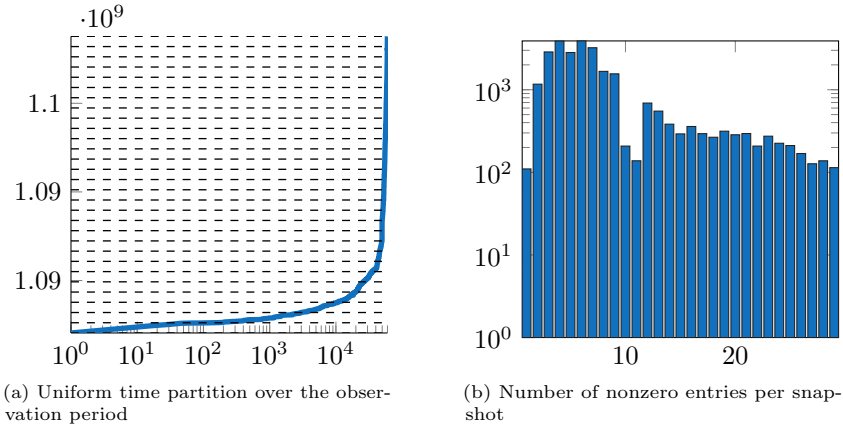}}\hfil
    \subfloat[Number of nonzero entries per snapshot\label{fig:messagenet-nnz}]{\input{nnz_message}}
    
    \caption{\texttt{CollegeMsg}~\cite{Panzarasa2009911} dataset ($n=1899$ nodes, $n_t-1=29$ snapshots). Left: the $30$ uniformly spaced breakpoints, normalized to $[0,10]$, that define the time windows; each dot on the message-count curve marks a window boundary. Right: histogram of the number of nonzero entries $\operatorname{nnz}(A_i)$ across the $29$ adjacency matrix snapshots, illustrating the strong variability in communication activity over time.}
    \label{fig:messagenet}
\end{figure}

Before running the full optimization, we validate the Krylov method for the matrix logarithm action described in Section~\ref{sec:discretizing_and_solving}. For each snapshot $A_i^\top$, we apply \texttt{logmv} with tolerance $10^{-6}$ and a maximum of $40$ Arnoldi iterations to compute $\log(I - a A_i^\top)\mathbf{e}$, where $\mathbf{e}$ is a fixed random Gaussian vector, and compare the result against the dense reference~\cite{MR3080997} $\texttt{logm}(I - aA_i^\top)\mathbf{e}$. The absolute and relative errors, together with the number of Krylov iterations $k$ required to meet the stopping criterion, are reported in Table~\ref{tab:krylov_logm_convergence}. Across all $29$ snapshots the method converges in at most $18$ iterations, and both absolute and relative errors remain well below $10^{-7}$, confirming that the Krylov-based procedure from Section~\ref{sec:discretizing_and_solving} is an effective method for the matrices arising in this dataset.
\begin{table}[htbp]
\caption{Convergence of the Krylov approximation of $\log(I - aA_i^\top)\allowbreak \mathbf{e}$ for each snapshot of the \texttt{Col\-lege\-Msg} \cite{Panzarasa2009911} dataset. $k$: number of Arnoldi iterations; absolute and relative errors against the dense reference. Tolerance $10^{-6}$, maximum $40$ iterations.}
\label{tab:krylov_logm_convergence}
\centering
\setlength{\tabcolsep}{0.8em}
\begin{tabular}{rrrr@{\qquad}rrrr}
\toprule
$i$ & $k$ & Abs.\ error & Rel.\ error & $i$ & $k$ & Abs.\ error & Rel.\ error \\
\midrule
 1 &  6 & $4.78\times10^{-11}$ & $1.75\times10^{-10}$ & 16 &  9 & $2.48\times10^{-9}$  & $2.53\times10^{-9}$  \\
 2 & 11 & $4.39\times10^{-9}$  & $2.41\times10^{-9}$  & 17 &  8 & $2.11\times10^{-9}$  & $3.96\times10^{-9}$  \\
 3 & 17 & $1.39\times10^{-8}$  & $3.89\times10^{-9}$  & 18 &  8 & $4.89\times10^{-10}$ & $9.41\times10^{-10}$ \\
 4 & 15 & $2.18\times10^{-8}$  & $6.01\times10^{-9}$  & 19 &  9 & $6.61\times10^{-10}$ & $1.38\times10^{-9}$  \\
 5 & 14 & $1.84\times10^{-8}$  & $7.69\times10^{-9}$  & 20 & 10 & $2.58\times10^{-9}$  & $2.76\times10^{-9}$  \\
 6 & 18 & $4.48\times10^{-8}$  & $4.61\times10^{-9}$  & 21 &  9 & $4.60\times10^{-10}$ & $5.98\times10^{-10}$ \\
 7 & 14 & $1.26\times10^{-8}$  & $3.66\times10^{-9}$  & 22 &  8 & $1.74\times10^{-10}$ & $4.54\times10^{-10}$ \\
 8 &  8 & $2.02\times10^{-9}$  & $1.65\times10^{-9}$  & 23 &  8 & $5.13\times10^{-9}$  & $5.36\times10^{-9}$  \\
 9 & 13 & $5.38\times10^{-9}$  & $3.16\times10^{-9}$  & 24 &  8 & $8.96\times10^{-10}$ & $1.04\times10^{-9}$  \\
10 &  7 & $1.06\times10^{-9}$  & $2.82\times10^{-9}$  & 25 &  8 & $3.30\times10^{-8}$  & $9.92\times10^{-9}$  \\
11 &  6 & $1.82\times10^{-10}$ & $1.09\times10^{-9}$  & 26 &  8 & $3.97\times10^{-10}$ & $8.93\times10^{-10}$ \\
12 & 10 & $4.63\times10^{-9}$  & $6.72\times10^{-9}$  & 27 &  8 & $9.75\times10^{-10}$ & $2.20\times10^{-9}$  \\
13 &  9 & $1.76\times10^{-9}$  & $2.99\times10^{-9}$  & 28 &  6 & $1.17\times10^{-9}$  & $5.05\times10^{-9}$  \\
14 &  8 & $4.11\times10^{-9}$  & $7.28\times10^{-9}$  & 29 &  7 & $8.05\times10^{-11}$ & $4.53\times10^{-10}$ \\
15 &  8 & $3.26\times10^{-9}$  & $3.90\times10^{-9}$  &    &    &                      &                      \\
\bottomrule
\end{tabular}
\end{table}
We perform a similar validation for the Fr\'echet-derivative action used in the gradient evaluation,
\[
\mathrm{L}_{\log}(A_i^\top,\boldsymbol{\lambda}\mathbf{r}^\top)\mathbf{v}.
\]
For each snapshot $A_i^\top$, we set $M_i = I-aA_i$, generate fixed random Gaussian vectors $\boldsymbol{\lambda},\mathbf{r},\mathbf{v}\in\mathbb{R}^n$, and compare the Krylov approximation produced by the rational Krylov approximation~\eqref{eq:rat_krylov} against the dense reference $\mathbf{y}_{\mathrm{ref}} = \mathrm{L}_{\log}(M_i^\top,\boldsymbol{\lambda}\mathbf{r}^\top)\mathbf{v}$, 
computed with \texttt{logm\_frechet\_pade}~\cite{MR3080997}. For the Fr\'echet action we now use the rational Krylov routine with a fixed set of $11$ poles (obtained by AAA on the $[1 - a \max_t \rho(A(t)), 1 + a \max_t \rho(A(t))]$ interval). The absolute and relative errors are reported in Table~\ref{tab:krylov_frechet_partial}. Across all $29$ snapshots, the relative errors range from $1.07\times10^{-14}$ to $1.60\times10^{-7}$, confirming that the rational Krylov Fr\'echet action is highly accurate for gradient evaluations in this dataset.
\begin{table}[htbp]
\caption{Convergence of the rational Krylov approximation of $\mathrm{L}_{\log}(A_i^\top,\boldsymbol{\lambda}\mathbf{r}^\top)\mathbf{v}$ for each snapshot of the \texttt{CollegeMsg}~\cite{Panzarasa2009911} dataset. Absolute and relative errors against the dense reference. A fixed set of $11$ poles is used for all snapshots.}
\label{tab:krylov_frechet_partial}
\centering
\setlength{\tabcolsep}{1em}
\begin{tabular}{rrr@{\qquad}rrr}
\toprule
$i$ & Abs.\ error & Rel.\ error & $i$ & Abs.\ error & Rel.\ error \\
\midrule
1 & $1.06\times10^{-12}$ & $1.07\times10^{-14}$ & 16 & $1.45\times10^{-11}$ & $1.60\times10^{-13}$ \\
2 & $4.05\times10^{-9}$ & $2.91\times10^{-11}$ & 17 & $6.60\times10^{-12}$ & $6.99\times10^{-14}$ \\
3 & $7.80\times10^{-6}$ & $1.60\times10^{-7}$ & 18 & $6.87\times10^{-12}$ & $7.01\times10^{-14}$ \\
4 & $1.14\times10^{-6}$ & $5.76\times10^{-9}$ & 19 & $7.23\times10^{-12}$ & $7.63\times10^{-14}$ \\
5 & $2.08\times10^{-7}$ & $1.50\times10^{-9}$ & 20 & $2.97\times10^{-9}$ & $3.25\times10^{-11}$ \\
6 & $9.17\times10^{-6}$ & $3.41\times10^{-8}$ & 21 & $1.73\times10^{-11}$ & $1.98\times10^{-13}$ \\
7 & $1.51\times10^{-7}$ & $8.91\times10^{-10}$ & 22 & $6.46\times10^{-12}$ & $5.82\times10^{-14}$ \\
8 & $7.82\times10^{-12}$ & $8.33\times10^{-14}$ & 23 & $2.13\times10^{-11}$ & $3.56\times10^{-13}$ \\
9 & $4.78\times10^{-8}$ & $3.22\times10^{-10}$ & 24 & $1.16\times10^{-11}$ & $1.79\times10^{-13}$ \\
10 & $6.14\times10^{-12}$ & $5.90\times10^{-14}$ & 25 & $2.02\times10^{-10}$ & $2.13\times10^{-12}$ \\
11 & $6.61\times10^{-12}$ & $6.35\times10^{-14}$ & 26 & $6.31\times10^{-12}$ & $6.54\times10^{-14}$ \\
12 & $7.08\times10^{-10}$ & $8.97\times10^{-12}$ & 27 & $4.74\times10^{-12}$ & $5.19\times10^{-14}$ \\
13 & $1.26\times10^{-11}$ & $1.54\times10^{-13}$ & 28 & $6.29\times10^{-12}$ & $6.27\times10^{-14}$ \\
14 & $6.80\times10^{-12}$ & $8.21\times10^{-14}$ & 29 & $6.79\times10^{-12}$ & $6.36\times10^{-14}$ \\
15 & $7.08\times10^{-12}$ & $7.86\times10^{-14}$ & & & \\
\bottomrule
\end{tabular}
\end{table}

We now consider the general problem using the \texttt{CollegeMsg} temporal network with approximately 1,899 nodes and 29 snapshots. The adjacency matrices are interpolated at a uniform time step $\Delta t = 0.1$ across the normalized time domain, yielding $116$ discrete time levels for optimization. Following Remark~\ref{rem:log-uniqueness}, the edge attenuation parameter is set to $a = 1/(2\rho_{\max})$, where $\rho_{\max} = \max_k \rho(A_k)$ is the maximum spectral radius. The temporal decay rate is $b = 0.85$, and the control regularization weight is $\alpha = 1$. For the Krylov approximations, we use tolerance $10^{-6}$ and maximum 40 iterations for polynomial actions on state/costate, and a fixed set of 11 poles (chosen via AAA) for rational Krylov Fréchet actions. To decide the set of edges of the temporal network we can modify, we first build the sparsity pattern $P$ as the union of all temporal adjacency patterns, then we restrict it to only the edges incident to 50 randomly selected nodes, creating $814$ controllable entries per time step. Feasibility is enforced via the Katz-based linear constraint from Proposition~\ref{prop:relaxed_constraint}, ensuring $\rho(A(t) + U(t)) < 1/a$ at all times through convex quadratic projections with numerical margin $10^{-6}$. To decide the target state, we first simulate the uncontrolled receive centrality $\mathbf{r}^{\mathrm{nc}}(T)$ to terminal time $T = 10$. The target state $\mathbf{r}^*$ is a perturbed version: for each of the 50 controllable nodes, we set $r^*_i = \max(2 \cdot r^{\mathrm{nc}}_i \cdot \xi, 1)$ where $\xi \sim \mathrm{U}(0,1)$ is uniform random, creating a non-trivial steering objective.

We employ the Nesterov accelerated projected-gradient method with monotone restart from Algorithm~\ref{alg:nesterov_pg}. The algorithm runs for up to $2000$ iterations with fixed step size $\eta = 0.10$, momentum coefficient and the three stopping criteria: $\|\nabla_U J\| \le 10^{-6}$, relative step size $\le 10^{-8}$, or relative objective change $\le 10^{-8}$. Monotone restart rejects momentum if an extrapolated gradient step increases the objective, enhancing robustness.

At each iteration, the algorithm computes the state trajectory (forward Euler), costate trajectory (backward adjoint, reverse time), and gradient via matrix-free Krylov evaluations. The control gradient is computed employing the Fréchet action via rational Krylov with AAA poles on $[1-a \max_{t}\rho(A(t)),1+a \max_{t}\rho(A(t))] \approx [a, 2]$. Feasibility projections via quadratic programming are applied to each time-step control after the gradient step. The accuracy of the Krylov approximations (as validated in Table~\ref{tab:krylov_frechet_partial}) ensures that repeated state/costate/gradient evaluations remain stable throughout the optimization.

Figure~\ref{fig:nesterov_krylov} shows the convergence history of the Nesterov accelerated projected-gradient solver on the full temporal network problem. 
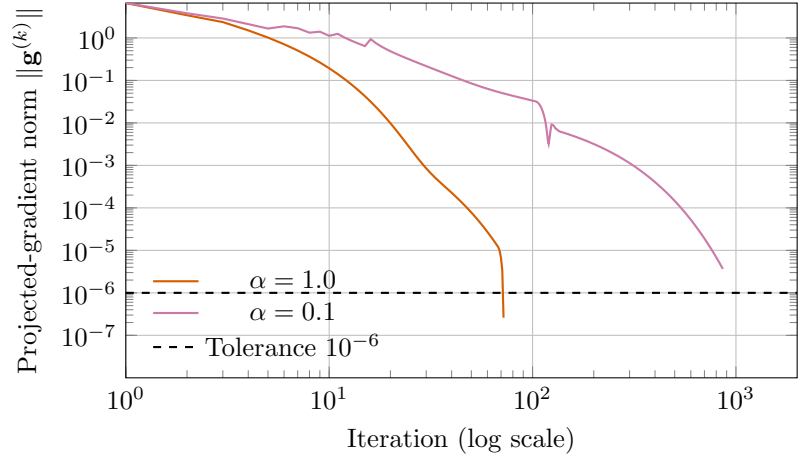
\begin{figure}[htb]
    \centering
    \begin{tikzpicture}
\begin{axis}[
    width=0.8\columnwidth,
    height=0.50\columnwidth,
    xlabel={Iteration (log scale)},
    ylabel={Projected-gradient norm $\|\mathbf{g}^{(k)}\|$},
    legend pos=north east,
    grid=major,
    ymode=log,
    xmode=log,
    xmin=1,
    xmax=2000,
    xtick={1, 10, 100, 1000},
    ytick={1e-7, 1e-6, 1e-5, 1e-4, 1e-3, 1e-2, 1e-1, 1e0},
    ymin=1e-8,
    ymax=6.62148712351667789732800884e+00,
    legend columns=1,
    legend style={
        at={(0.4, 0.018)},
        anchor=south east,
        draw=none,
        fill=none,
    }
]
\addplot[
    blind1,
    thick,
    mark=none,
] table[
    x index=0,
    y index=1,
] {converge_nesterov_krylov_newtransp.dat};
\addlegendentry{$\alpha = 1.0$}
\addplot[
    blind2,
    thick,
    mark=none,
] table[
    x index=0,
    y index=1,
] {convergence_nesterov_krylov_alpha01.dat};
\addlegendentry{$\alpha = 0.1$}
\addplot[
    style=dashed,
    color=black,
    thick,
] coordinates {(1,1e-6) (2000,1e-6)};
\addlegendentry{Tolerance $10^{-6}$}
\end{axis}
\end{tikzpicture}
\caption{Algorithm~\ref{alg:nesterov_pg} convergence on the \texttt{CollegeMsg} temporal network control problem.}
    \label{fig:nesterov_krylov}
\end{figure}
We note that even in the case of a larger-scale problem the behavior with respect to the regularization parameter is analogous to what was seen in the cases in Section~\ref{sec:small_examples}, the number of necessary iterations increases as $\alpha$ decreases, and the method declares convergence at iteration 862 with a projected gradient norm $\approx 3.663 \times 10^{-6}$, and a total convergence cost of $J \approx 3.758922$. The gradient tolerance is not satisfied, but the difference between two successive iterates has dropped below the threshold of $10^{-7}$ reaching $9.9 \times 10^{-8}$ and causing convergence to be declared.

\section{Conclusion and future perspectives}\label{sec:conclusion_and_future_perspectives}

In this work we formulated and analyzed a pattern-constrained optimal control problem for dynamic receive centrality on temporal networks. Using the Pontryagin Maximum Principle, we derived state, costate, and gradient conditions for both the logarithmic model and its linearized counterpart, and we introduced a practical projection strategy based on a sufficient linear constraint that preserves well-posedness of the logarithmic dynamics. On the algorithmic side, we adopted a Nesterov accelerated projected-gradient scheme with monotone restart, and we combined it with matrix-free Krylov approximations for both matrix-logarithm actions and Fr\'echet-derivative actions. The numerical results show that this combination is effective on both synthetic and real temporal data: the optimization procedure remains stable under pattern constraints, and Krylov-based evaluations provide accurate approximations at moderate subspace dimensions, making repeated state/costate/gradient evaluations computationally viable.

Several extensions can follow from the present framework. A first direction is a sharper theoretical analysis of convergence for the projected accelerated method in the nonconvex setting induced by the logarithmic dynamics, including quantitative guarantees for restart strategies. A second direction is to further reduce the cost of the Fr\'echet term by exploiting block, low-rank, or recycling Krylov spaces across consecutive time steps and optimization iterations. A third direction is to investigate adaptive and higher-order time integration tailored to piecewise-smooth temporal networks, balancing accuracy and runtime in large horizons. From a modeling perspective, it is also of interest to consider multi-objective formulations that jointly steer receive centrality and other dynamic-network functionals, as well as robust or stochastic variants that account for uncertainty in temporal edges. Finally, extending the framework to distributed and parallel implementations would be essential for very large networks, where sparse linear algebra and matrix-function actions become the dominant computational bottleneck.

\backmatter

\bmhead{Acknowledgements}

All the authors are members of the INdAM GNCS and acknowledge the MUR Excellence Department Project awarded to the Department of Mathematics, University of Pisa, CUP I57G22000700001.

\bmhead{Funding}
This research has been partially funded by the Italian Ministry of University and Research (MUR) through the PRIN 2022 ``MOLE: Manifold constrained Optimization and LEarning'',  code: 2022ZK5ME7 MUR D.D. financing decree n. 20428 of November 6th, 2024 (CUP B53C24006410006) and Fondo Italiano per la Scienza (FIS2023-01334) advanced grant ``ADvanced numerical Approaches for MUltiscale Systems with uncertainties"- ADAMUS.

\bmhead{Data availability}  The code to reproduce the numerical examples is available on the GitHub repository \href{https://github.com/Cirdans-Home/optimal-network-control}{Cirdans-Home/optimal-network-control}.

\section*{Declarations}

\bmhead{Ethics approval}
Not applicable
\bmhead{Conflict of interest}
 The authors declare no competing interests.

\bibliography{bibdynet}

\end{document}

%% file: phonecall2_1.tex
\begin{tikzpicture}[scale=0.5]
\Vertex[x=2.00,y=6.50,size=0.20,RGB,color={51,51,255}]{1}
\Vertex[x=2.00,y=5.50,size=0.20,RGB,color={51,51,255}]{2}
\Vertex[x=2.00,y=4.50,size=0.20,RGB,color={51,51,255}]{3}
\Vertex[x=1.33,y=3.50,size=0.20,RGB,color={51,51,255}]{4}
\Vertex[x=2.67,y=3.50,size=0.20,RGB,color={51,51,255}]{5}
\Vertex[x=0.80,y=2.50,size=0.20,RGB,color={51,51,255}]{6}
\Vertex[x=1.60,y=2.50,size=0.20,RGB,color={51,51,255}]{7}
\Vertex[x=2.40,y=2.50,size=0.20,RGB,color={51,51,255}]{8}
\Vertex[x=3.20,y=2.50,size=0.20,RGB,color={51,51,255}]{9}
\Vertex[x=0.44,y=1.50,size=0.20,RGB,color={51,51,255}]{10}
\Vertex[x=0.89,y=1.50,size=0.20,RGB,color={51,51,255}]{11}
\Vertex[x=1.33,y=1.50,size=0.20,RGB,color={51,51,255}]{12}
\Vertex[x=1.78,y=1.50,size=0.20,RGB,color={51,51,255}]{13}
\Vertex[x=2.22,y=1.50,size=0.20,RGB,color={51,51,255}]{14}
\Vertex[x=2.67,y=1.50,size=0.20,RGB,color={51,51,255}]{15}
\Vertex[x=3.11,y=1.50,size=0.20,RGB,color={51,51,255}]{16}
\Vertex[x=3.56,y=1.50,size=0.20,RGB,color={51,51,255}]{17}
\Edge[bend=-30](1)(3)
\Edge[bend=30](2)(17)
\end{tikzpicture}

%% file: phonecall2_2.tex
\begin{tikzpicture}[scale=0.5]
\Vertex[x=2.00,y=6.00,size=0.20,RGB,color={51,51,255}]{1}
\Vertex[x=2.00,y=5.00,size=0.20,RGB,color={51,51,255}]{2}
\Vertex[x=2.00,y=4.00,size=0.20,RGB,color={51,51,255}]{3}
\Vertex[x=1.33,y=3.00,size=0.20,RGB,color={51,51,255}]{4}
\Vertex[x=2.67,y=3.00,size=0.20,RGB,color={51,51,255}]{5}
\Vertex[x=0.80,y=2.00,size=0.20,RGB,color={51,51,255}]{6}
\Vertex[x=1.60,y=2.00,size=0.20,RGB,color={51,51,255}]{7}
\Vertex[x=2.40,y=2.00,size=0.20,RGB,color={51,51,255}]{8}
\Vertex[x=3.20,y=2.00,size=0.20,RGB,color={51,51,255}]{9}
\Vertex[x=0.44,y=1.00,size=0.20,RGB,color={51,51,255}]{10}
\Vertex[x=0.89,y=1.00,size=0.20,RGB,color={51,51,255}]{11}
\Vertex[x=1.33,y=1.00,size=0.20,RGB,color={51,51,255}]{12}
\Vertex[x=1.78,y=1.00,size=0.20,RGB,color={51,51,255}]{13}
\Vertex[x=2.22,y=1.00,size=0.20,RGB,color={51,51,255}]{14}
\Vertex[x=2.67,y=1.00,size=0.20,RGB,color={51,51,255}]{15}
\Vertex[x=3.11,y=1.00,size=0.20,RGB,color={51,51,255}]{16}
\Vertex[x=3.56,y=1.00,size=0.20,RGB,color={51,51,255}]{17}
\Edge(3)(4)
\end{tikzpicture}

%% file: phonecall2_3.tex
\begin{tikzpicture}[scale=0.5]
\Vertex[x=2.00,y=6.00,size=0.20,RGB,color={51,51,255}]{1}
\Vertex[x=2.00,y=5.00,size=0.20,RGB,color={51,51,255}]{2}
\Vertex[x=2.00,y=4.00,size=0.20,RGB,color={51,51,255}]{3}
\Vertex[x=1.33,y=3.00,size=0.20,RGB,color={51,51,255}]{4}
\Vertex[x=2.67,y=3.00,size=0.20,RGB,color={51,51,255}]{5}
\Vertex[x=0.80,y=2.00,size=0.20,RGB,color={51,51,255}]{6}
\Vertex[x=1.60,y=2.00,size=0.20,RGB,color={51,51,255}]{7}
\Vertex[x=2.40,y=2.00,size=0.20,RGB,color={51,51,255}]{8}
\Vertex[x=3.20,y=2.00,size=0.20,RGB,color={51,51,255}]{9}
\Vertex[x=0.44,y=1.00,size=0.20,RGB,color={51,51,255}]{10}
\Vertex[x=0.89,y=1.00,size=0.20,RGB,color={51,51,255}]{11}
\Vertex[x=1.33,y=1.00,size=0.20,RGB,color={51,51,255}]{12}
\Vertex[x=1.78,y=1.00,size=0.20,RGB,color={51,51,255}]{13}
\Vertex[x=2.22,y=1.00,size=0.20,RGB,color={51,51,255}]{14}
\Vertex[x=2.67,y=1.00,size=0.20,RGB,color={51,51,255}]{15}
\Vertex[x=3.11,y=1.00,size=0.20,RGB,color={51,51,255}]{16}
\Vertex[x=3.56,y=1.00,size=0.20,RGB,color={51,51,255}]{17}
\Edge(3)(5)
\Edge(4)(6)
\end{tikzpicture}

%% file: phonecall2_4.tex
\begin{tikzpicture}[scale=0.5]
\Vertex[x=2.00,y=6.00,size=0.20,RGB,color={51,51,255}]{1}
\Vertex[x=2.00,y=5.00,size=0.20,RGB,color={51,51,255}]{2}
\Vertex[x=2.00,y=4.00,size=0.20,RGB,color={51,51,255}]{3}
\Vertex[x=1.33,y=3.00,size=0.20,RGB,color={51,51,255}]{4}
\Vertex[x=2.67,y=3.00,size=0.20,RGB,color={51,51,255}]{5}
\Vertex[x=0.80,y=2.00,size=0.20,RGB,color={51,51,255}]{6}
\Vertex[x=1.60,y=2.00,size=0.20,RGB,color={51,51,255}]{7}
\Vertex[x=2.40,y=2.00,size=0.20,RGB,color={51,51,255}]{8}
\Vertex[x=3.20,y=2.00,size=0.20,RGB,color={51,51,255}]{9}
\Vertex[x=0.44,y=1.00,size=0.20,RGB,color={51,51,255}]{10}
\Vertex[x=0.89,y=1.00,size=0.20,RGB,color={51,51,255}]{11}
\Vertex[x=1.33,y=1.00,size=0.20,RGB,color={51,51,255}]{12}
\Vertex[x=1.78,y=1.00,size=0.20,RGB,color={51,51,255}]{13}
\Vertex[x=2.22,y=1.00,size=0.20,RGB,color={51,51,255}]{14}
\Vertex[x=2.67,y=1.00,size=0.20,RGB,color={51,51,255}]{15}
\Vertex[x=3.11,y=1.00,size=0.20,RGB,color={51,51,255}]{16}
\Vertex[x=3.56,y=1.00,size=0.20,RGB,color={51,51,255}]{17}
\Edge[bend=30](1)(17)
\Edge(2)(3)
\Edge(4)(7)
\Edge(5)(8)
\Edge(6)(10)
\end{tikzpicture}

%% file: phonecall2_5.tex
\begin{tikzpicture}[scale=0.5]
\Vertex[x=2.00,y=6.00,size=0.20,RGB,color={51,51,255}]{1}
\Vertex[x=2.00,y=5.00,size=0.20,RGB,color={51,51,255}]{2}
\Vertex[x=2.00,y=4.00,size=0.20,RGB,color={51,51,255}]{3}
\Vertex[x=1.33,y=3.00,size=0.20,RGB,color={51,51,255}]{4}
\Vertex[x=2.67,y=3.00,size=0.20,RGB,color={51,51,255}]{5}
\Vertex[x=0.80,y=2.00,size=0.20,RGB,color={51,51,255}]{6}
\Vertex[x=1.60,y=2.00,size=0.20,RGB,color={51,51,255}]{7}
\Vertex[x=2.40,y=2.00,size=0.20,RGB,color={51,51,255}]{8}
\Vertex[x=3.20,y=2.00,size=0.20,RGB,color={51,51,255}]{9}
\Vertex[x=0.44,y=1.00,size=0.20,RGB,color={51,51,255}]{10}
\Vertex[x=0.89,y=1.00,size=0.20,RGB,color={51,51,255}]{11}
\Vertex[x=1.33,y=1.00,size=0.20,RGB,color={51,51,255}]{12}
\Vertex[x=1.78,y=1.00,size=0.20,RGB,color={51,51,255}]{13}
\Vertex[x=2.22,y=1.00,size=0.20,RGB,color={51,51,255}]{14}
\Vertex[x=2.67,y=1.00,size=0.20,RGB,color={51,51,255}]{15}
\Vertex[x=3.11,y=1.00,size=0.20,RGB,color={51,51,255}]{16}
\Vertex[x=3.56,y=1.00,size=0.20,RGB,color={51,51,255}]{17}
\Edge(5)(9)
\Edge(6)(11)
\Edge(7)(12)
\Edge(8)(14)
\end{tikzpicture}

%% file: phonecall2_6.tex
\begin{tikzpicture}[scale=0.5]
\Vertex[x=2.00,y=6.00,size=0.20,RGB,color={51,51,255}]{1}
\Vertex[x=2.00,y=5.00,size=0.20,RGB,color={51,51,255}]{2}
\Vertex[x=2.00,y=4.00,size=0.20,RGB,color={51,51,255}]{3}
\Vertex[x=1.33,y=3.00,size=0.20,RGB,color={51,51,255}]{4}
\Vertex[x=2.67,y=3.00,size=0.20,RGB,color={51,51,255}]{5}
\Vertex[x=0.80,y=2.00,size=0.20,RGB,color={51,51,255}]{6}
\Vertex[x=1.60,y=2.00,size=0.20,RGB,color={51,51,255}]{7}
\Vertex[x=2.40,y=2.00,size=0.20,RGB,color={51,51,255}]{8}
\Vertex[x=3.20,y=2.00,size=0.20,RGB,color={51,51,255}]{9}
\Vertex[x=0.44,y=1.00,size=0.20,RGB,color={51,51,255}]{10}
\Vertex[x=0.89,y=1.00,size=0.20,RGB,color={51,51,255}]{11}
\Vertex[x=1.33,y=1.00,size=0.20,RGB,color={51,51,255}]{12}
\Vertex[x=1.78,y=1.00,size=0.20,RGB,color={51,51,255}]{13}
\Vertex[x=2.22,y=1.00,size=0.20,RGB,color={51,51,255}]{14}
\Vertex[x=2.67,y=1.00,size=0.20,RGB,color={51,51,255}]{15}
\Vertex[x=3.11,y=1.00,size=0.20,RGB,color={51,51,255}]{16}
\Vertex[x=3.56,y=1.00,size=0.20,RGB,color={51,51,255}]{17}
\Edge(7)(13)
\Edge(8)(15)
\Edge(9)(16)
\end{tikzpicture}

%% file: phonecall2_7.tex
\begin{tikzpicture}[scale=0.5]
\Vertex[x=2.00,y=6.00,size=0.20,RGB,color={51,51,255}]{1}
\Vertex[x=2.00,y=5.00,size=0.20,RGB,color={51,51,255}]{2}
\Vertex[x=2.00,y=4.00,size=0.20,RGB,color={51,51,255}]{3}
\Vertex[x=1.33,y=3.00,size=0.20,RGB,color={51,51,255}]{4}
\Vertex[x=2.67,y=3.00,size=0.20,RGB,color={51,51,255}]{5}
\Vertex[x=0.80,y=2.00,size=0.20,RGB,color={51,51,255}]{6}
\Vertex[x=1.60,y=2.00,size=0.20,RGB,color={51,51,255}]{7}
\Vertex[x=2.40,y=2.00,size=0.20,RGB,color={51,51,255}]{8}
\Vertex[x=3.20,y=2.00,size=0.20,RGB,color={51,51,255}]{9}
\Vertex[x=0.44,y=1.00,size=0.20,RGB,color={51,51,255}]{10}
\Vertex[x=0.89,y=1.00,size=0.20,RGB,color={51,51,255}]{11}
\Vertex[x=1.33,y=1.00,size=0.20,RGB,color={51,51,255}]{12}
\Vertex[x=1.78,y=1.00,size=0.20,RGB,color={51,51,255}]{13}
\Vertex[x=2.22,y=1.00,size=0.20,RGB,color={51,51,255}]{14}
\Vertex[x=2.67,y=1.00,size=0.20,RGB,color={51,51,255}]{15}
\Vertex[x=3.11,y=1.00,size=0.20,RGB,color={51,51,255}]{16}
\Vertex[x=3.56,y=1.00,size=0.20,RGB,color={51,51,255}]{17}
\Edge(9)(17)
\end{tikzpicture}

%% file: phonecalltwodesiredstate1.tex
\begin{tikzpicture}

\begin{axis}[%
width=0.169\columnwidth,
height=0.158\columnwidth,
at={(0\columnwidth,0\columnwidth)},
scale only axis,
xmin=1,
xmax=17,
xlabel style={font=\color{white!15!black}\small},
xlabel={Node},
ymin=1,
ymax=1.839361122715,
title={$\mathbf{r}(7)$},
axis background/.style={fill=white},
xmajorgrids,
ymajorgrids
]
\addplot [color=mycolor2, only marks, mark=triangle, mark options={solid, mycolor2}, forget plot]
  table[row sep=crcr]{%
1	1.06829963546012\\
2	1.08051668798367\\
3	1.11343340557298\\
4	1.09808860304556\\
5	1.16481104853837\\
6	1.16238899198996\\
7	1.26948112002542\\
8	1.26591347459247\\
9	1.49032829337727\\
10	1.05898265952191\\
11	1.09789296212984\\
12	1.09302943352027\\
13	1.15796683561631\\
14	1.09218908245625\\
15	1.15712458632702\\
16	1.14741308730373\\
17	1.37187845340705\\
};
\end{axis}

\begin{axis}[%
width=0.169\columnwidth,
height=0.158\columnwidth,
at={(0.222\columnwidth,0\columnwidth)},
scale only axis,
xmin=1,
xmax=17,
xlabel style={font=\color{white!15!black}\small},
xlabel={Node},
ymin=1,
ymax=1.839361122715,
title={$\mathbf{r}^*$},
axis background/.style={fill=white},
xmajorgrids,
ymajorgrids
]
\addplot [color=mycolor1, only marks, mark=o, mark options={solid, mycolor1}, forget plot]
  table[row sep=crcr]{%
1	1.06829963546012\\
2	1.08051668798367\\
3	1.2\\
4	1.2\\
5	1.16481104853837\\
6	1.16238899198996\\
7	1.26948112002542\\
8	1.26591347459247\\
9	1.49032829337727\\
10	1.05898265952191\\
11	1.09789296212984\\
12	1.09302943352027\\
13	1.15796683561631\\
14	1.09218908245625\\
15	1.15712458632702\\
16	1.14741308730373\\
17	1.37187845340705\\
};
\end{axis}

\end{tikzpicture}%

%% file: solution-phone.tex
\begin{tikzpicture}
\begin{groupplot}[
    group style={
        group size=1 by 3,
        vertical sep=0.03\columnwidth
    },
    width=\columnwidth,
    height=0.32\columnwidth,
    xmin=0, xmax=18,
    ymin=1, ymax=1.6,
    xmajorgrids,
    ymajorgrids,
    ylabel style={font=\color{mycolor3}},
    xlabel style={font=\color{mycolor3}},
]

\nextgroupplot[
    ylabel={$\alpha = 1$},
    xticklabels=\empty,
    legend columns=3,
    legend style={
        at={(0.5,1.15)},
        anchor=south,
        draw=none,
        fill=none,
        legend cell align=left,
        align=left,
    }
]

\addplot [color=red, only marks, mark=x] table[row sep=crcr]{%
1	1.0786393426712\\
2	1.09483968305955\\
3	1.15077162735902\\
4	1.13542756265054\\
5	1.20772072879306\\
6	1.20621519720419\\
7	1.34233865477421\\
8	1.33309250046405\\
9	1.58955606030043\\
10	1.0685413831856\\
11	1.12197935679257\\
12	1.11368375507223\\
13	1.20109087773216\\
14	1.11247260053834\\
15	1.19961396273128\\
16	1.18895703590209\\
17	1.42181246474763\\
};
\addlegendentry{Obtained State}

\addplot [color=mycolor1, only marks, mark=o] table[row sep=crcr]{%
1	1.08079942814544\\
2	1.09680542787917\\
3	1.2\\
4	1.2\\
5	1.20136727832821\\
6	1.19805985080723\\
7	1.3230535789279\\
8	1.31809809215451\\
9	1.58964650365395\\
10	1.06975582918861\\
11	1.11787941863531\\
12	1.11136089697095\\
13	1.18426661615855\\
14	1.11020674225078\\
15	1.18310096778931\\
16	1.17118427594065\\
17	1.44412871354738\\
};
\addlegendentry{Desired State}

\addplot [color=mycolor2, only marks, mark=triangle] table[row sep=crcr]{%
1	1.08079942814544\\
2	1.09680542787917\\
3	1.13971649104835\\
4	1.12018985464132\\
5	1.20136727832821\\
6	1.19805985080723\\
7	1.3230535789279\\
8	1.31809809215451\\
9	1.58964650365395\\
10	1.06975582918861\\
11	1.11787941863531\\
12	1.11136089697095\\
13	1.18426661615855\\
14	1.11020674225078\\
15	1.18310096778931\\
16	1.17118427594065\\
17	1.44412871354738\\
};
\addlegendentry{Uncontrolled}

\nextgroupplot[
    ylabel={$\alpha = 0.5$},
    xticklabels=\empty,
]

\addplot [color=red, only marks, mark=x] table[row sep=crcr]{%
1	1.08088042483289\\
2	1.09724699340787\\
3	1.15915004577134\\
4	1.14410246955396\\
5	1.20768529007561\\
6	1.20733533176274\\
7	1.34143812752833\\
8	1.32994677141666\\
9	1.59024820824711\\
10	1.06819523483222\\
11	1.12133468745286\\
12	1.11259957424824\\
13	1.19883227731458\\
14	1.11169852473112\\
15	1.19765221528784\\
16	1.18766682405383\\
17	1.42484815018526\\
};

\addplot [color=mycolor1, only marks, mark=o] table[row sep=crcr]{%
1	1.08079942814544\\
2	1.09680542787917\\
3	1.2\\
4	1.2\\
5	1.20136727832821\\
6	1.19805985080723\\
7	1.3230535789279\\
8	1.31809809215451\\
9	1.58964650365395\\
10	1.06975582918861\\
11	1.11787941863531\\
12	1.11136089697095\\
13	1.18426661615855\\
14	1.11020674225078\\
15	1.18310096778931\\
16	1.17118427594065\\
17	1.44412871354738\\
};

\addplot [color=mycolor2, only marks, mark=triangle] table[row sep=crcr]{%
1	1.08079942814544\\
2	1.09680542787917\\
3	1.13971649104835\\
4	1.12018985464132\\
5	1.20136727832821\\
6	1.19805985080723\\
7	1.3230535789279\\
8	1.31809809215451\\
9	1.58964650365395\\
10	1.06975582918861\\
11	1.11787941863531\\
12	1.11136089697095\\
13	1.18426661615855\\
14	1.11020674225078\\
15	1.18310096778931\\
16	1.17118427594065\\
17	1.44412871354738\\
};

\nextgroupplot[
    ylabel={$\alpha = 0.05$},
    xlabel={Node},
]

\addplot [color=red, only marks, mark=x] table[row sep=crcr]{%
1	1.08452597902433\\
2	1.10090222548038\\
3	1.18968033345151\\
4	1.18105450746469\\
5	1.2044561731658\\
6	1.20688685537181\\
7	1.33376051671088\\
8	1.3190314024434\\
9	1.59058129755791\\
10	1.06555609015619\\
11	1.11592389194697\\
12	1.10777238002573\\
13	1.18603318851323\\
14	1.11061295569294\\
15	1.18932420644004\\
16	1.17802750660452\\
17	1.43669093696263\\
};

\addplot [color=mycolor1, only marks, mark=o] table[row sep=crcr]{%
1	1.08079942814544\\
2	1.09680542787917\\
3	1.2\\
4	1.2\\
5	1.20136727832821\\
6	1.19805985080723\\
7	1.3230535789279\\
8	1.31809809215451\\
9	1.58964650365395\\
10	1.06975582918861\\
11	1.11787941863531\\
12	1.11136089697095\\
13	1.18426661615855\\
14	1.11020674225078\\
15	1.18310096778931\\
16	1.17118427594065\\
17	1.44412871354738\\
};

\addplot [color=mycolor2, only marks, mark=triangle] table[row sep=crcr]{%
1	1.08079942814544\\
2	1.09680542787917\\
3	1.13971649104835\\
4	1.12018985464132\\
5	1.20136727832821\\
6	1.19805985080723\\
7	1.3230535789279\\
8	1.31809809215451\\
9	1.58964650365395\\
10	1.06975582918861\\
11	1.11787941863531\\
12	1.11136089697095\\
13	1.18426661615855\\
14	1.11020674225078\\
15	1.18310096778931\\
16	1.17118427594065\\
17	1.44412871354738\\
};

\end{groupplot}
\end{tikzpicture}

%% file: nnz_message.tex
\definecolor{mycolor1}{rgb}{0.06600,0.44300,0.74500}%
\definecolor{mycolor2}{rgb}{0.12941,0.12941,0.12941}%
\begin{tikzpicture}

\begin{axis}[%
width=0.32\columnwidth,
height=0.3\columnwidth,
at={(0\columnwidth,0\columnwidth)},
scale only axis,
bar shift auto,
ymode=log,
xmin=0.6,
xmax=29.4,
ymin=1,
ymax=3898,
axis background/.style={fill=white}
]
\addplot[ybar, bar width=0.8, fill=mycolor1, draw=mycolor2, area legend] table[row sep=crcr] {%
1	110\\
2	1168\\
3	2868\\
4	3881\\
5	2822\\
6	3898\\
7	3224\\
8	1669\\
9	1558\\
10	208\\
11	138\\
12	690\\
13	551\\
14	383\\
15	292\\
16	359\\
17	294\\
18	266\\
19	315\\
20	285\\
21	295\\
22	208\\
23	274\\
24	225\\
25	211\\
26	169\\
27	127\\
28	138\\
29	114\\
};
\addplot[forget plot, color=mycolor2] table[row sep=crcr] {%
0.6	0\\
29.4	0\\
};
\end{axis}
\end{tikzpicture}%